\documentclass[12pt]{amsart}

\usepackage{amsmath,amssymb,amsbsy,amsfonts,amsthm,latexsym,enumerate,
	amsopn,amstext,amsxtra,euscript,amscd,mathrsfs}
\usepackage{amsmath,amssymb,amsbsy,amsfonts,latexsym,amsopn,amstext,cite,
	amsxtra,euscript,amscd,bm,bbm}
\usepackage{mathtools}
\usepackage{epigraph}
\usepackage{url}
\usepackage{enumitem}

\begin{document}
	
%%%%%%%%%%%%%%%%%%%%%%%%%%%%%%%%%%%%%%%%%%%%%%%%%%%%%%%%
%%%%%%%%%%%%%%%%%%%%%%%%%%%%%%%%%%%%%%%%%%%%%%%%%%%%%%%%
%%%%%%%%%%%%%%%%%%%%%%%%%%%%%%%%%%%%%%%%%%%%%%%%%%%%%%%%
%%%%%%%%%%%%%%%%%%%%%%%%%%%%%%%%%%%%%%%%%%%%%%%%%%%%%%%%

%%%%%%%   STANDARD STUFF %%%%%%%%%

%%%%%%%%%%%%%%%%%%%%%%%%%%%%%%%%%%%%%%%%%%%%%%%%%%%%%%%%
%%%%%%%%%%%%%%%%%%%%%%%%%%%%%%%%%%%%%%%%%%%%%%%%%%%%%%%%
%%%%%%%%%%%%%%%%%%%%%%%%%%%%%%%%%%%%%%%%%%%%%%%%%%%%%%%%
%%%%%%%%%%%%%%%%%%%%%%%%%%%%%%%%%%%%%%%%%%%%%%%%%%%%%%%%

%  use the AMS-Euler Fraktur fonts
%%%%%%%%%%%%%%%%%%%%%%%%%%%%%%%%%%
\newfont{\teneufm}{eufm10}
\newfont{\seveneufm}{eufm7}
\newfont{\fiveeufm}{eufm5}
%%%%%%%%%%%%%%%%%%%%%%%%%%%%%%%%%
%
%  allow automatic size selection in math mode
%
%%%%%%%%%%%%%%%%%%%%%%%%%%%%%%%%%
\newfam\eufmfam
\textfont\eufmfam=\teneufm \scriptfont\eufmfam=\seveneufm
\scriptscriptfont\eufmfam=\fiveeufm
%%%%%%%%%%%%%%%%%%%%%%%%%%%%%%%%%
%
%  \frak works on a single symbol at a time...
%
\def\frak#1{{\fam\eufmfam\relax#1}}
%

%%%%%%%%%%%%%%%%%%%  bbb-matter

\def\bbbr{{\rm I\!R}} %reelle Zahlen
\def\bbbm{{\rm I\!M}}
\def\bbbn{{\rm I\!N}} %natuerliche Zahlen
\def\bbbf{{\rm I\!F}}
\def\bbbh{{\rm I\!H}}
\def\bbbk{{\rm I\!K}}
\def\bbbp{{\rm I\!P}}
\def\bbbone{{\mathchoice {\rm 1\mskip-4mu l} {\rm 1\mskip-4mu l}
		{\rm 1\mskip-4.5mu l} {\rm 1\mskip-5mu l}}}
\def\bbbc{{\mathchoice {\setbox0=\hbox{$\displaystyle\rm C$}\hbox{\hbox
				to0pt{\kern0.4\wd0\vrule height0.9\ht0\hss}\box0}}
		{\setbox0=\hbox{$\textstyle\rm C$}\hbox{\hbox
				to0pt{\kern0.4\wd0\vrule height0.9\ht0\hss}\box0}}
		{\setbox0=\hbox{$\scriptstyle\rm C$}\hbox{\hbox
				to0pt{\kern0.4\wd0\vrule height0.9\ht0\hss}\box0}}
		{\setbox0=\hbox{$\scriptscriptstyle\rm C$}\hbox{\hbox
				to0pt{\kern0.4\wd0\vrule height0.9\ht0\hss}\box0}}}}
\def\bbbq{{\mathchoice {\setbox0=\hbox{$\displaystyle\rm
				Q$}\hbox{\raise
				0.15\ht0\hbox to0pt{\kern0.4\wd0\vrule height0.8\ht0\hss}\box0}}
		{\setbox0=\hbox{$\textstyle\rm Q$}\hbox{\raise
				0.15\ht0\hbox to0pt{\kern0.4\wd0\vrule height0.8\ht0\hss}\box0}}
		{\setbox0=\hbox{$\scriptstyle\rm Q$}\hbox{\raise
				0.15\ht0\hbox to0pt{\kern0.4\wd0\vrule height0.7\ht0\hss}\box0}}
		{\setbox0=\hbox{$\scriptscriptstyle\rm Q$}\hbox{\raise
				0.15\ht0\hbox to0pt{\kern0.4\wd0\vrule height0.7\ht0\hss}\box0}}}}
\def\bbbt{{\mathchoice {\setbox0=\hbox{$\displaystyle\rm
				T$}\hbox{\hbox to0pt{\kern0.3\wd0\vrule height0.9\ht0\hss}\box0}}
		{\setbox0=\hbox{$\textstyle\rm T$}\hbox{\hbox
				to0pt{\kern0.3\wd0\vrule height0.9\ht0\hss}\box0}}
		{\setbox0=\hbox{$\scriptstyle\rm T$}\hbox{\hbox
				to0pt{\kern0.3\wd0\vrule height0.9\ht0\hss}\box0}}
		{\setbox0=\hbox{$\scriptscriptstyle\rm T$}\hbox{\hbox
				to0pt{\kern0.3\wd0\vrule height0.9\ht0\hss}\box0}}}}
\def\bbbs{{\mathchoice
		{\setbox0=\hbox{$\displaystyle     \rm S$}\hbox{\raise0.5\ht0\hbox
				to0pt{\kern0.35\wd0\vrule height0.45\ht0\hss}\hbox
				to0pt{\kern0.55\wd0\vrule height0.5\ht0\hss}\box0}}
		{\setbox0=\hbox{$\textstyle        \rm S$}\hbox{\raise0.5\ht0\hbox
				to0pt{\kern0.35\wd0\vrule height0.45\ht0\hss}\hbox
				to0pt{\kern0.55\wd0\vrule height0.5\ht0\hss}\box0}}
		{\setbox0=\hbox{$\scriptstyle      \rm S$}\hbox{\raise0.5\ht0\hbox
				to0pt{\kern0.35\wd0\vrule height0.45\ht0\hss}\raise0.05\ht0\hbox
				to0pt{\kern0.5\wd0\vrule height0.45\ht0\hss}\box0}}
		{\setbox0=\hbox{$\scriptscriptstyle\rm S$}\hbox{\raise0.5\ht0\hbox
				to0pt{\kern0.4\wd0\vrule height0.45\ht0\hss}\raise0.05\ht0\hbox
				to0pt{\kern0.55\wd0\vrule height0.45\ht0\hss}\box0}}}}
\def\bbbz{{\mathchoice {\hbox{$\sf\textstyle Z\kern-0.4em Z$}}
		{\hbox{$\sf\textstyle Z\kern-0.4em Z$}}
		{\hbox{$\sf\scriptstyle Z\kern-0.3em Z$}}
		{\hbox{$\sf\scriptscriptstyle Z\kern-0.2em Z$}}}}
\def\ts{\thinspace}

\newtheorem{thm}{Theorem}
\newtheorem{lem}{Lemma}
\newtheorem{lemma}[thm]{Lemma}
\newtheorem{prop}{Proposition}
\newtheorem{proposition}[thm]{Proposition}
\newtheorem{theorem}[thm]{Theorem}
\newtheorem{cor}[thm]{Corollary}
\newtheorem{corollary}[thm]{Corollary}
\newtheorem*{thank}{\ \ \ \bf Acknowledgment}
\newtheorem{prob}{Problem}
\newtheorem{problem}[prob]{Problem}
\newtheorem*{ques}{Question}
\newtheorem*{rem}{Remarks}
%% DEFINITIONS

\numberwithin{equation}{section}
\numberwithin{thm}{section}

\def\squareforqed{\hbox{\rlap{$\sqcap$}$\sqcup$}}
\def\qed{\ifmmode\squareforqed\else{\unskip\nobreak\hfil
		\penalty50\hskip1em\null\nobreak\hfil\squareforqed
		\parfillskip=0pt\finalhyphendemerits=0\endgraf}\fi}

%%%%%%%%%%%%%%%%%%%%%%%%%
% Alphabet calligraphie %
%%%%%%%%%%%%%%%%%%%%%%%%%
\def\cA{{\mathcal A}}
\def\cB{{\mathcal B}}
\def\cC{{\mathcal C}}
\def\cD{{\mathcal D}}
\def\cE{{\mathcal E}}
\def\cF{{\mathcal F}}
\def\cG{{\mathcal G}}
\def\cH{{\mathcal H}}
\def\cI{{\mathcal I}}
\def\cJ{{\mathcal J}}
\def\cK{{\mathcal K}}
\def\cL{{\mathcal L}}
\def\cM{{\mathcal M}}
\def\cN{{\mathcal N}}
\def\cO{{\mathcal O}}
\def\cP{{\mathcal P}}
\def\cQ{{\mathcal Q}}
\def\cR{{\mathcal R}}
\def\cS{{\mathcal S}}
\def\cT{{\mathcal T}}
\def\cU{{\mathcal U}}
\def\cV{{\mathcal V}}
\def\cW{{\mathcal W}}
\def\cX{{\mathcal X}}
\def\cY{{\mathcal Y}}
\def\cZ{{\mathcal Z}}
\def\pgdc{\textrm{gcd}}
\newcommand{\rmod}[1]{\: \mbox{mod} \: #1}

\def\Nm{{\mathrm{Nm}}}

\def\Tr{{\mathrm{Tr}}}

\def\epp{\mathbf{e}_{p-1}}

\def\ind{\mathop{\mathrm{ind}}}

\def\mand{\qquad \mbox{and} \qquad}

\newcommand*{\QEDA}{\hfill\ensuremath{\blacksquare}}

\newcommand{\commH}[1]{\marginpar{%
		\begin{color}{magenta}
			\vskip-\baselineskip %raise the marginpar a bit
			\raggedright\footnotesize
			\itshape\hrule \smallskip H: #1\par\smallskip\hrule\end{color}}}

\newcommand{\commI}[1]{\marginpar{%
		\begin{color}{blue}
			\vskip-\baselineskip %raise the marginpar a bit
			\raggedright\footnotesize
			\itshape\hrule \smallskip I: #1\par\smallskip\hrule\end{color}}}

%%%%%%%%%%%%%%%%%%%%%%%%%%%%%%%%%%%%%%%%%%%%%%%%%%%%%%%%
%%%%%%%%%%%%%%%%%%%%%%%%%%%%%%%%%%%%%%%%%%%%%%%%%%%%%%%%
%%%%%%%%%%%%%%%%%%%%%%%%%%%%%%%%%%%%%%%%%%%%%%%%%%%%%%%%
%%%%%%%%%%%%%%%%%%%%%%%%%%%%%%%%%%%%%%%%%%%%%%%%%%%%%%%%

%%%%%%%  END OF STANDARD STUFF %%%%%%%%%

%%%%%%%%%%%%%%%%%%%%%%%%%%%%%%%%%%%%%%%%%%%%%%%%%%%%%%%%
%%%%%%%%%%%%%%%%%%%%%%%%%%%%%%%%%%%%%%%%%%%%%%%%%%%%%%%%
%%%%%%%%%%%%%%%%%%%%%%%%%%%%%%%%%%%%%%%%%%%%%%%%%%%%%%%%
%%%%%%%%%%%%%%%%%%%%%%%%%%%%%%%%%%%%%%%%%%%%%%%%%%%%%%%
\newcommand{\ignore}[1]{}

\hyphenation{re-pub-lished}

\parskip 1.5 mm
\parindent 8 pt

\def\GL{\operatorname{GL}}
\def\SL{\operatorname{SL}}
\def\PGL{\operatorname{PGL}}
\def\PSL{\operatorname{PSL}}
\def\li{\operatorname{li}}

\def\vec#1{\mathbf{#1}}

\def \F{{\mathbb F}}
\def \E{{\mathbb E}}
\def \K{{\mathbb K}}
\def \Z{{\mathbb Z}}
\def \N{{\mathbb N}}
\def \Q{{\mathbb Q}}
\def \T{{\mathbb T}}
\def \C {{\mathbb C}}
\def \I {{\mathbbm 1 }}
\def \R{{\mathbb R}}
\def\Fp{\F_p}
\def \fp{\Fp^*}

\def \Rc{{\mathcal R}}
\def \Qc{{\mathcal Q}}
\def \Ec{{\mathcal E}}

\def \DN{D_N}
\def\va{\mbox{\bf{a}}}

\def\Kc{\,{\mathcal K}}
\def\Ic{\,{\mathcal I}}

\def\\{\cr}
\def\({\left(}
\def\){\right)}
\def\fl#1{\left\lfloor#1\right\rfloor}
\def\rf#1{\left\lceil#1\right\rceil}

\def\Ln#1{\mbox{\rm {Ln}}\,#1}

\def \nd {\, | \hspace{-1.2mm}/\,}

\def\e{\mathbf{e}}

\def\ep{\mathbf{e}_p}
\def\eq{\mathbf{e}_q}

\def\wt#1{\mbox{\rm {wt}}\,#1}

\def\Mob{M{\"o}bius }

%===========================================================
\newcommand{\tend}[3][]{\xrightarrow[#2\to#3]{#1}}
\newcommand{\spdj}{\perp_{\mbox{\scriptsize sp}}}
\newcommand{\egdef}{:=}
%=====================================================
\newcommand{\Supp}{\Sigma}%notation pour l'ensemble des coef non nuls
%\newcommand{\Id}{\mathop{\mbox{Id}}}	
%\DeclareMathOperator{\Id}{Id}
%=======================================================
%%%%%%%%%%%%%%%%%%%%%%%%%%%%%%%%%%%%%%%%%%%%%%%%%%%%%%%%%%%%%
\newcommand{\setdef}{\stackrel {\rm {def}}{=}}
\newcommand{\spa}{\textrm{span}}
\newcommand{\ds}{\displaystyle}
\newcommand{\script}{\scriptstyle}

\newcommand*{\Scale}[2][4]{\scalebox{#1}{$#2$}}%
\newcommand*{\Resize}[2]{\resizebox{#1}{!}{$#2$}}
\newcommand*{\Sup}{\Resize{0.7cm}{sup}}
%%%%%%%%%%%%%%%%%%%%%%%%%%%%%%%%%%%%%%%%%%%%%%%%%%%%%%%%%%%%%%%%

%%%%%%%%%%%%%%%  Topmatter %%%%%%%%%%%%%%%%%%

\title[with oscillation method.]{Simple proof of Bourgain bilinear ergodic theorem and its extension to polynomials and polynomials in primes}
%of Bikhoff Ergodic theorem and

\author[E. H. El Abdalaoui]{El Houcein El Abdalaoui}
\address{University of Rouen Normandy,
	LMRS UMR  6085 CNRS-UNIV, 
	F76801 Saint-{\'E}tienne-du-Rouvray, France}
\email{elhoucein.elabdalaoui@univ-rouen.fr}

\date{\today}
\subjclass[2010]{Primary: 37A30;  Secondary: 28D05, 5D10, 11B30, 11N37, 37A45}
%\date{}

\pagenumbering{arabic}
\begin{abstract} We first present a modern simple  proof of the classical ergodic Birkhoff's theorem and Bourgain's homogeneous bilinear ergodic theorem. This proof used the simple fact that the shift map on integers has a simple Lebesgue spectrum. As a consequence, we establish that the homogeneous bilinear ergodic averages along polynomials and polynomials in primes converge almost everywhere, that is, for any invertible measure preserving transformation $T$, acting on a probability space $(X, \cB, \mu)$, for any $f \in L^r(X,\mu)$ , $g \in L^{r'}(X,\mu)$ such that $\frac{1}{r}+\frac{1}{r'}= 1$, for any non-constant polynomials $P(n),Q(n), n \in \Z$, taking integer values, and for almost all $x \in X$, we have,
	$$\frac{1}{N}\sum_{n=1}^{N}f(T^{P(n)}x) g(T^{Q(n)}x),$$
and 
$$\frac{1}{\pi_N}\sum_{\overset{p \leq N}{p\textrm{~~prime}}}f(T^{P(p)}x) g(T^{Q(p)}x),$$	
converge. Here $\pi_N$ is the number of prime in $[1,N]$. 
\end{abstract}

\maketitle

\epigraph{What is hardest of all? That which seems most simple: to see with your eyes what is before your eyes.}{\textit{ Goethe }}
%\hrule
%Dedicated to the memory of Professor J. Bourgain.\\
%\hrule
\section{Introduction} 

The classical ergodic theorem has many proofs which use in some sense of-course the classical ergodic maximal inequality. The known elegant proof of the ergodic maximal inequality is due to Garcia \cite{Garcia}. This proof is reproduced by almost all authors in any introduction book on ergodic theory. There is also an ``Easy and nearly simultaneous proofs of the Ergodic Theorem and Maximal Ergodic Theorem" established by  M. Keane and K. Petersen \cite{KP}. Moreover, using the non standard ideas of Kamae to proof ergodic theorem, Y. Katznelson and B. Weiss produced a combinatorics proof of it \cite{KW}. In \cite[Theorem 10.28, p.110]{Nadkarni1}, M. G. Nadkarni gives a measure free proof of Birkhoff's theorem, and state a descriptive version of Rhoklin lemma.

Here,  our aim is to produce a simple proof of Birkhoff theorem using the oscillating method. We will further produce a simple proof of \linebreak Bourgain bilinear ergodic theorem \cite{Bourgain-D}. This later theorem gives an affirmative answer to the question raised by H. Furstenberg \cite[Question 1. p. 96]{Fbook}. Precisely, it assert that the homogeneous bilinear ergodic average converge almost everywhere. For a finitary simple proof of it, we refer to \cite{elabdal}. Subsequently, we will extend Bourgain homogeneous bilinear ergodic theorem to polynomials and polynomials in primes.
% by establishing that the  bilinear ergodic averages along  converge almost surely.

Obviously, Bourgain bilinear ergodic theorem is a generalization of Birkhoff theorem. But, as we will see, the proof of it depend heavily on Birkhoff theorem.

Let us point out also that the Birkhoff ergodic theorem and the Hopf ergodic maximal inequality are equivalent. For more details, we refer to \cite[Chap.~1]{Garcia} and the references therein.
      
\section{Set-Up and tools} 

Let $(X,\cA,T,\mu)$ a dynamical system, that is, $(X,\cA,\mu)$ is a probability space and $T$ is a measure preserving transformation. $T$ is said to be ergodic if the measure of any invariant set $A$ (i.e. $\mu(A \Delta T^{-1}(A))=0$) is $0$ or $1$. It is well-know that the reduction to the ergodic case can be used by applying the ergodic decomposition. Therefore, in many cases, it is suffices to study the ergodic case. We denote by $L^2(X,\mu)$ the space of square integrable functions and by $\cI$  the $\sigma$-algebra of invariant sets.
%$L_0^2(X,\mu)$ the subspace of $L^2$ functions with mean zero, that is, $$L_0^2(X,\mu) \setdef\Big\{f \in L^2(X,\mu)~~:~~\int f(x) d\mu(x)=0\Big\}.$$ 

In this setting, Birkhoff ergodic theorem assert that for any $f \in L^1(X,\mu)$, for almost all $x \in X$, we have
\begin{eqnarray}\label{BK}
\frac{1}{N}\sum_{1}^{N} f(T^nx) \tend{n}{+\infty}\E(f|\cI).
\end{eqnarray}

The Bourgain bilinear ergodic theorem say that for any $f,g \in L^2(X,\mu)$, for almost all $x \in X$,  for any $a, b \in \Z$,  

\begin{equation}\label{BD}
	\frac{1}{N}\sum_{n=1}^{N}f(T^{an}x) g(T^{bn}x) {\textrm{~~~~converge. }}
\end{equation}

The prime ergodic theorem established by Bourgain \cite{B3} assert that for any $f \in L^2(X)$, for almost all $x \in X$,
\begin{equation}\label{BP}
\frac{1}{\pi_N}\sum_{\overset{p \leq N}{p\textrm{~prime}}}f(T^{p}x)  {\textrm{~~~~converge.}}
\end{equation}

Bourgain prime ergodic theorem \eqref{BP} was strengthened to $L^r$, $r>1$, by Wierdl. But, as pointed out by Mirek and Trojan \cite{MirekT}, the reader should be made aware that there is a technical gap in \cite[p.331]{Wierdl}. Therefore, the approach of Wierdl should be combined with that of Mirek-Trojan \cite{MirekT} or Cuny-Weber \cite{Cuny-W}. Later, Nair extended \eqref{BP} \cite{Nair} by proving that for any $r>1$, for any non-constant polynomial $Q$ mapping the naturel numbers to themseleves, for any $f \in L^r(X,\mu)$,

\begin{equation}\label{NP}
\frac{1}{\pi_N}\sum_{\overset{p \leq N}{p\textrm{~prime}}}f(T^{Q(p)}x)  
\end{equation}
converge almost surely in $x$ with respect to $\mu$.

 Here, we will prove the following theorem.
\begin{thm}\label{BPPET}Let $P(n), Q(n), n \in \Z$ be a non-constant polynomials taking integer values. Let $(X,\cA,T,\mu)$ be a measure preserving dynamical system. Then,  for any $f \in L^r(X,\mu), g \in L^{r'}(X,\mu)$ such that $ \frac{1}{r}+\frac{1}{r'}=1$,  the bilinear averages 
$$\frac{1}{N}\sum_{n=1}^{N}f(T^{P(n)}x) g(T^{Q(n)}x),$$
and the prime bilinear averages 
$$\frac{1}{\pi_N}\sum_{\overset{p \leq N}{p\textrm{~~prime}}}f(T^{P(p)}x) g(T^{Q(p)}x),$$	
converge almost everywhere in $x$ with respect to $\mu$.
\end{thm}
The proof of Theorem \ref{BPPET} is postponed to section \ref{ProofBPPET}. But, for the moment, let us point out that the most important result needed is the following strong maximal ergodic inequality,  which may be of independent interest. 
\begin{thm}\label{MPPET}Let $P(n), Q(n), n \in \Z$ be a non-constant polynomials taking integer values. Let $(X,\cA,T,\mu)$ be a measure preserving dynamical system. Then,  for any $f \in L^r(X,\mu), g \in L^{r'}(X,\mu)$ such that $\frac{1}{r}+\frac{1}{r'}=1$, we have
$$\Bigg\|\sup_{ N  \geq 1} \Big|\frac{1}{N} \sum_{n=1}^{N}f(T^{P(n)}x) g(T^{Q(n)}x)\Bigg\|_1 \leq C_{r} \big\|f\big\|_r \big\|f\big\|_{r'}.$$
\end{thm}

We start by noticing that, obviously, we have
$$f= \E(f|\cI)+f-\E(f|\cI),$$
Furthermore, it is easy to see that \eqref{BK} holds for the function $\E(f|\cI)$. We thus need to establish only that the convergence holds for a functions of the form $f-\E(f|\cI)$. Notice further that for such functions the limit is zero.
  
The Hopf ergodic maximal inequality state that for any $f \in L^1(X,\mu)$, for any $\lambda>0$, the maximal function $\ds M(f)(x) \setdef \sup_{N \geq 1}\Big|\frac{1}{N}\sum_{n=1}^{N}f \circ T^n(x)\Big|$ satisfy
$$\mu\Biggr\{x~~:~~M(f)(x)>\lambda\Biggl\}\leq\frac{\big\|f\big\|_1}{\lambda}.$$ 
It follows from the Hopf ergodic maximal inequality that it is suffices to see that the convergence holds almost everywhere  for a dense set. Indeed, let $\varepsilon >0$ and $f \in L_0^2(X,\mu)$, assume that there exist a function $g$ for which the converge almost everywhere holds and $\|f-g\| \leq \varepsilon^2$. Then,
$$\mu\Biggr\{x~~:~~M(f-g)(x)> \varepsilon\Biggl\}\leq\varepsilon.$$
Whence, there exist a Borel set $X_\epsilon$ with measure great than $1-\varepsilon$ for which we have, for any $x \in X_\epsilon$,
$$\limsup \Big|\frac{1}{N}\sum_{n=1}^{N}(f-g) \circ T^n(x)\Big| \leq \varepsilon,$$
We thus get
$$\mu\Biggr\{\limsup \Big|\frac{1}{N}\sum_{n=1}^{N}f\circ T^n(x)\Big| \geq  \varepsilon\Biggl\} \leq \varepsilon.$$
Take a sequence $(\varepsilon_n)$ such that $\sum \varepsilon_n <+\infty$ and apply Borel-Cantelli to see that for almost all $x \in X$, 

$$\limsup \Big|\frac{1}{N}\sum_{n=1}^{N}f\circ T^n(x)\Big|=0.$$
%We thus get that the invariant set 
%$$\Biggr\{x~~:~~ \limsup \Big|\frac{1}{N}\sum_{n=1}^{N}f\circ T^n(x)\Big|
%\leq \varepsilon
%\Biggl\}$$
%has a full measure. 
Hence, for almost all $x \in X$, we have 
$$\frac{1}{N}\sum_{n=1}^{N}f\circ T^n(x) \tend{N}{+\infty}0.$$
Notice that we have also proved that if the Hopf ergodic maximal inequality holds then the  family of functions for which the convergence almost everywhere holds is closed. 
\paragraph{\textbf{Spectral measure  and spectral tools.}}
The notion  of  spectral measure for sequence goes back to Wiener who introduce it in
his 1933 book \cite{Wiener}. More precisely, Wiener considers the space $S$ of
complex bounded sequences $g =(g_{n})_{n \in \N}$ such that
\begin{equation}\label{Sspace}
\lim_{N \longrightarrow +\infty}
\frac{1}{N}\sum_{n=1}^{N}g_{n+k}\overline{g}_{n}=F(k)
\end{equation}
exists for each integer $k \in \N$. The sequence $F (k)$ can be extended to negative
integers by setting
\[
F(-k)=\overline{F (k)}.
\]
It is well known that $F$ is positive definite on $\Z$ and therefore (by
Herglotz-Bochner theorem) there exists a unique positive finite measure $\sigma_g$ on
the circle such that the Fourier
coefficients of $\sigma_g$ are given by the sequence $F$.
Formally, we have
\[
\widehat{\sigma_g}(k)\stackrel{\rm{def}}{=}
\int_{-\pi}^{\pi} e^{-ikx} d\sigma_{g }(x) = F(k).
\]
The measure $\sigma_g $ is called the {\em spectral measure of the sequence $g$}.

This is can be linked to the spectral theory of the dynamical system. Indeed, if $T$ is an invertible measure preserving transformation of the $\sigma$-finite measure space $(X, \cA, m)$, then 
$T$ induces an
operator $U_T$ in $L^p(X)$ via $f \mapsto U_{T} (f)= f \circ T$
called Koopman operator.
For  $p=2$ this operator is unitary and its spectral
resolution induces a spectral decomposition of $L^{2} (X)$ \cite{parry} (see also \cite{Nadkarni} and \cite{aaronson}):
\begin{align}\label{spec-d}
L^2(X)=\bigoplus_{n=0}^{+\infty}C(f_i) {\rm {~~and~~}}
\sigma_{f_1}\gg \sigma_{f_2}\gg\cdots
\end{align}
where
\begin{itemize}
	\item $\{f_{i} \}_{i=1}^{+\infty}$ is a family of functions in $L^{2} (X)$;
	\item $C(f)\setdef \overline{\rm {span}}\{U_T^n(f): n \in \Z\}$ is the cyclic
	space generated by $f \in L^2(X)$;
	\item  $\sigma_f$ is the {\em spectral measure} on the circle generated by $f$
	via the Bochner-Herglotz relation
	\begin{equation}\label{fspmeasure}
	\widehat{\sigma_f}(n)=<U_T^nf,f>=\int_X f \circ T^n(x) \overline{f}(x)d\mu(x);
	\end{equation}
	\item for any two measures on the circle $\alpha$ and $\beta$,
	$\alpha \gg \beta$ means $\beta $ is absolutely continuous  with respect to
	$\alpha$: for any Borel set, $\alpha(A)=0 \Longrightarrow \beta(A)=0$.
	The two measures $\alpha$ and  $\beta$ are equivalent if  and only if
	$\alpha\gg\beta$ and $\beta\gg\alpha$. 
	%We will denote measure equivalence by
	%$\alpha \sim \beta$.
\end{itemize}
As a nice exercise, it can be seen that if the map $T$ is acting ergodically on a probability space then for almost all $x \in X$, $\sigma_f$ is the weak limit of the following sequence of finite measures on the circle 
\begin{equation*}
\sigma_N =\frac{1}{2\pi}\Big|\frac{1}{\sqrt{N}}\sum_{n=1}^{N}f(T^nx)e^{in\theta}\Big|^2d\theta ,
\end{equation*}
that is, for almost all $x\in X$, the sequence $(f(T^nx))$ is in the space $S$.
  
The spectral theorem ensures that the spectral decomposition \eqref{spec-d} is unique up to
isomorphisms.
The {\em maximal spectral type} of $T$ is the equivalence class of the Borel
measure $\sigma_{f_1}$. The multiplicity function
$\cM_{T} : \T \longrightarrow \{1,2,\cdots,\} \cup \{+\infty\}$ is defined
$\sigma_{f_1}$ a.e. and
\begin{eqnarray*}
&\cM_T(z)=\ds \sum_{n=1}^{+\infty}\I_{Y_j}(z), \\ &\quad {\rm  where}, \ Y_1=\T \ {\rm and}\
Y_j={\rm{~supp~}}\Big(\frac{d\sigma_{f_j}}{d\sigma_{f_1}}\Big) \quad \forall j \geq 2.
\end{eqnarray*}
An integer $n \in \{1,2,\cdots,\} \cup \{+\infty\}$ is called an essential value of
$M_T$ if $\sigma_{f_1}\{ z \in \T : M_T(z)=n\}>0$. The multiplicity
is uniform or homogeneous if there is only one essential value of $M_T$.
The essential supremum of $M_T$ is called the maximal  spectral multiplicity of $T$.
The map $T$
\begin{itemize}
	\item has simple spectrum if $L^2(X)$ is reduced to a
	single cyclic space;
	\item has discrete spectrum if $L^2(X)$ has an
	orthonormal basis consisting of eigenfunctions of $U_T$
	(in this case $\sigma_{f_1}$ is a discrete measure);
	\item has Lebesgue spectrum (resp. absolutely continuous,
	singular spectrum) if
	$\sigma_{f_1}$ is equivalent  (resp. absolutely
	continuous, singular) to the Lebesgue measure.
\end{itemize}
The {\em reduced spectral type } of the dynamical system  is its spectral type on
the $L_0^2 (X)$ the space of square integrable functions with zero mean.
%\begin{def}\label{defsingular}
%	Two dynamical systems are called {\em {spectrally disjoint}} if their reduced
%	spectral types are mutually singular.
%\end{def}
%\newline

Here, we will use the fact that the shift map on $\Z$ $(S~~:~~n \mapsto n+1)$ acting on $\Z$ equipped with the counting measure has a simple Lebesgue spectrum.

 As customary, the Fourier transform of $ f \in \ell^2(\Z)$ is denoted by 
$$\widehat{f}(\theta)=\sum_{n \in \Z}f(n)e^{-in\theta}, ~~~~\forall \theta \in [-\pi,\pi),$$
and the Fourier transform of a function $F \in L^2([-\pi,\pi])$ is given by
$$\widehat{F}(n)=\frac{1}{2\pi}\int_{-\pi}^{\pi}F(\theta)e^{-in\theta} d\theta,  ~~~~\forall n \in \Z.$$ 
The convolution operator $*$ is given by 
$$f*g(j)=\sum_{x \in \Z} f(x)g(j-x), ~~~~~~~~~\forall f,g \in \ell^1(\Z).$$

It can be easily seen that the Fourier transform operator $\cF$ gives a spectral isomorphism between  $\ell^2(\Z)$ and $L^2([-\pi,\pi])$. This is one of fundamental ingredient in Bourgain oscillation method.

According to the spectral isomorphism, we have 
$$\cF^{-1}{\big(FG\big)}= \cF^{-1}(F)*\cF^{-1}(G),~~~~ \forall F,G \in L^2(\T).$$

Therefore the convolution can be seen as multiplication from spectral point of view. 

 \section{Classical proof of  Birkhoff's  ergodic theorem }
We start by noticing that the closer of the subspace $C=\Big\{g-g \circ T,~~~g \in L^2(X,\mu)\Big\}$ is dense in the space $\Big\{f-\E(f|\cA), f \in L^2(X,\mu)\Big\}$. Indeed, let $\phi$ be a continuous linear on $L^2(X,\mu)$. So, $\phi$ is given by the multiplication by some $h \in L^2(X,\mu)$ and we have
$$\phi(f)=\int f(x) \overline{h}(x)  d\mu(x).$$
Assuming that $\phi(g-g\circ T)=0$, for any $g \in L^2(X,\mu)$, we get 
$h=h \circ T^{-1}$, that is, $h$ is an invariant function. Hence, by ergodicity, $h$ is constant almost everywhere. We thus get, by the nice properties of conditional expectation, 
\begin{align*}
	\phi(f-\E(f|\cA))&=\int f(x) \overline{h} d\mu(x)-
	\int \E(f|\cA) \overline{h} d\mu(x)\\
	&= \int f(x) \overline{h} d\mu(x)-
	\int \E(f \overline{h}|\cA)  d\mu(x)\\
	&= \int f(x) \overline{h} d\mu(x)-
	\int f \overline{h}  d\mu(x)
\end{align*}
for all $f \in L_0^2(X,\mu)$. Whence $\overline{C}=\Big\{f-\E(f|\cA), f \in L^2(X,\mu)\Big\}.$
Now, it is straightforward that if $g \in L^{\infty}$ then, for almost all $x \in X$, we have
$$\frac{1}{N}\sum_{0}^{N-1} \big(g-g\circ T)(T^nx) \tend{n}{+\infty}0.$$

To finish the proof, we  notice that $L^{\infty}$ is dense in $L^1$.
\QEDA

\textbf{Remarks.}
\begin{enumerate}
\item  One can obtain the proof of ergodic Birkhoff theorem without using ergodic maximal ergodic if the $L^2$-convergence holds with some speed. Meanly,  if 
$$\Bigg\|\frac{1}{N}\sum_{n=1}^{N}f \circ T^n(.) -\E(f|\cI)\Bigg\|_2 \leq C_f. \psi(N),$$
and for any $\rho>1$, $\sum_{n \geq 1}\psi([\rho^n])<+\infty,$
To accomplish the proof in this case, it suffices to use Etamedi's trick \cite{Etamedi}.
%\item One can use most easily that $L^{\infty}$ is dense in $L^1$ and take $g \in  L^\infty$.
\item The Hopf ergodic maximal inequality take different forms in Analysis. It is known also as  Hard-Littlewood maximal inequality \cite[Theorem 8]{HL1}, \cite[Th. 326]{HL2} (for more historical background, see the introduction in \cite{Be}), Kolomogrov-Doob inequalities \cite[Chap. 14 p.137-138]{KD1}, \cite[Theorem 2.1., Theorem 2.2, p. 14-15]{KD2}, Carleson-Hunt maximal inequalities \cite{CH1}, \cite{CH2}, \cite{CH3}, \cite{CH4}. We refer further to \cite[Chap. 2]{Simon} 
	\end{enumerate}
We end this section by pointing out that here, by exploiting Etamedi's trick, we will  prove only that the almost everywhere convergence  holds for a sequences of the form $N_m=[\rho^m]$, $m \in \N$, for any $\rho >1$. Such sequence is denoted by $S_\rho$.

\section{Calder\'{o}n transference principle and the maximal ergodic inequalities for high dimension}
The sequence of complex number $(a_n)$ is said to be a good weight in $L^p(X,\mu)$, $p\geq1$ for linear case, if, for any $f \in L^p(X,\mu)$, 
the ergodic averages
$$\frac{1}{N}\sum_{j=1}^{N}a_jf(T^jx)$$
converges a.e. (almost everywhere). We further say that the maximal ergodic inequality holds in $L^p(X,\mu)$ for the linear case with weight $(a_n)$ if, 
for any $f \in L^p(X,\mu)$, the maximal function given by 
$$M(f)(x)=\sup_{N \geq 1}\Big|\frac{1}{N}\sum_{j=1}^{N}a_jf(T^jx)\Big|$$
satisfy the weak-type inequality 
$$\lambda \mu\Big\{x~~:~~M(f)(x)>\lambda \Big\} \leq C \big\|f\big\|_p,$$
for any $\lambda>0$ with $C$ is an absolutely constant.

It is well known that the classical maximal ergodic inequality (Hopf maximal inequality) is equivalent to the Birkhoff ergodic theorem \cite{Garcia}.

The previous notions can be extended in the usual manner to the multilinear case. Let $k \geq 2$ and $(T_i)_{i=1}^{k}$ be a maps on a probability space $(X,\cA,\mu)$. W thus say that $(a_n)$ 
is good weight in $L^{p_i}(X,\mu)$, $p_i\geq1$, $i=1,\cdots,k$, with 
$\sum_{i=1}^{k}\frac1{p_i}=1,$ if, for any $f_i \in L^{p_i}(X,\mu)$, $i=1,\cdots,k$,
the ergodic $k$-multilinear averages   
$$\frac1{N}\sum_{j=1}^{N}a_j\prod_{i=1}^{k}f_i(T_i^jx),$$
converges a.e.. The maximal multilinear ergodic inequality is said to hold in $L^{p_i}(X,\mu)$, $p_i\geq1$, $i=1,\cdots,k$, with 
$\sum_{i=1}^{k}\frac1{p_i}=1,$ if, for any $f_i \in L^{p_i}(X,\mu)$, $i=1,\cdots,k$, the maximal function given by 
$$M(f_1,\cdots,f_k)(x)=\sup_{N \geq 1}\Big|\frac{1}{N}\sum_{j=1}^{N}a_j\prod_{i=1}^{k}f_i(T_i^jx)\Big|$$
satisfy the weak-type inequality 
$$\lambda \mu\Big\{x~~:~~M(f)(x)>\lambda \Big\} \leq C \prod_{i=1}^{k}\big\|f_i\big\|_{p_i},$$
for any $\lambda>0$ with $C$ is an absolutely constant.

It is not known whether the classical maximal multilinear ergodic inequality ($a_n=1$, for each $n$) holds 
for the general case $n \geq 3$. Nevertheless, we have the following Calder\'{o}n transference principal in the homogeneous case. 
\begin{prop}\label{CalderonP} Let $(a_n)$ be a sequence of complex number and assume that for any $\phi,\psi \in \ell^2(\Z)$, we have
	$$\Big\|\sup_{ N  \geq 1}\Big|\frac1{N}\sum_{n=1}^{N}a_n \phi(j+n)\psi(j-n) 
	\Big|\Big\|_{\ell^1(\Z)}< C.
	\big\|\phi\big\|_{\ell^2(\Z)}\big\|\psi\big\|_{\ell^2(\Z)},$$
	where $C$ is an absolutely constant. Then, for any dynamical system $(X,\cA,T,\mu)$, for any $f,g \in L^2(X,\mu)$, we have 
	%Then, 
	%for any $\rho>1$, for any $\delta>0$, there exist $N_0 \in I_\rho $ independent of 
	%$f$ and $g$ such that 
	$$\Big\|\sup_{N \geq 1}\Big|\frac1{N}\sum_{n=1}^{N}a_n f(T^nx)g(T^{-n}x) \Big|\Big\|_1<
	C \big\|f\big\|_{2}\big\|g\big\|_{2},$$ 
\end{prop}
We further have
%In \cite{DemeterLT}, the authors  that the the classical maximal multilinear ergodic inequality holds for the 
%homogneous case provided $\frac{1}{p_{k+1}}<3/2$.\\
\begin{prop}\label{CalderonP2} Let $(a_n)$ be a sequence of complex number and  assume that for any $\phi,\psi \in \ell^2(\Z)$, for any $\lambda>0$, for any integer $J \geq 2$, we have
\begin{eqnarray*}
	&\Big|\Big\{1 \leq j \leq J~:~\sup_{ N  \geq 1}\Big|\frac1{N}\sum_{n=1}^{N}a_n \phi(j+n)\psi(j-n) 
	\Big|> \lambda \Big\}\Big| \\&<C \frac{\big\|\phi\big\|_{\ell^2(\Z)}\big\|\psi\big\|_{\ell^2(\Z)}}{\lambda},
	%\IEEEeqnarraynumspace
	%\label{Bobservation}
	\end{eqnarray*}
where $C$ is an absolutely constant. Then, for any dynamical system $(X,\cA,T,\mu)$, for any $f,g \in L^2(X,\mu)$, we have 
	%Then, 
	%for any $\rho>1$, for any $\delta>0$, there exist $N_0 \in I_\rho $ independent of 
	%$f$ and $g$ such that 
	$$\mu\Big\{x \in X~~:~~\sup_{N \geq 1}\Big|\frac1{N}\sum_{n=1}^{N}a_n f(T^nx)g(T^{-n}x) \Big| > \lambda \Big\}<
	C \frac{\big\|f\big\|_{2}.\big\|g\big\|_{2}}{\lambda}.$$ 
\end{prop}

The proof of Propositions \ref{CalderonP}  and  \ref{CalderonP2} is similar to that given in \cite[p. 135]{Thouvenot}; we include it for the reader's convenience.

{\textbf{Proof of  Propositions \ref{CalderonP}  and  \ref{CalderonP2}.}}
Let $\bar{N} \in \N$ and $J$ a positive integer such that $J \gg \bar{N}$. For any $x \in X,$ put
\begin{gather}
\phi_x(n)=\begin{cases}
f(T^nx) \: &\textrm{ if }  0 \leq |n| \leq J,\\
0  & \textrm{if not,}\\
\end{cases} 
\end{gather}
and,
\begin{gather}
\psi_x(n)=\begin{cases}
g(T^nx) \: &\textrm{if }  0 \leq |n| \leq J,\\
 0 &\textrm{if not.}\\
\end{cases} 
\end{gather}

Therefore $\phi_x, \psi_x \in \ell^2(\Z)$, we further have
\begin{gather}
	\big\|\phi_x\big\|_{\ell^2(\Z)}^2=\sum_{-J}^{J}\big|f(T^nx)|^2, \; \; \textrm{and} \; \; 
	\big\|\psi_x\big\|_{\ell^2(\Z)}^2=\sum_{-J}^{J}\big|g(T^{-n}x)|^2.
\end{gather}
We thus get, by our assumption, 
\begin{align}
&\sum_{-(J-\bar{N})}^{J-\bar{N}} \sup_{ \bar{N} \geq N  \geq 1}
\Big|\frac1{N}\sum_{n=1}^{N}a_n \phi_x(j+n)\psi_x(j-n) 
\Big| \nonumber \\
&\leq C \sqrt{\sum_{-J}^{J}\big|f(T^nx)\big|^2}
\sqrt{\sum_{-J}^{J}\big|g(T^{-n}x)\big|^2}.
\end{align}
Integrating and remembering that $T$ is a measure preserving transformation, we obtain
\begin{align}\label{inequa:1}
&\sum_{-(J-\bar{N})}^{J-\bar{N}}\int 
\sup_{ \bar{N} \geq N  \geq 1}
\Big|\frac1{N}\sum_{n=1}^{N}a_n f(T^{n})g(T^{-n}) 
\Big| d\mu(x) \nonumber\\
&\leq C \int \sqrt{\sum_{-J}^{J}\big|f(T^{n}x)\big|^2} \sqrt{\sum_{-J}^{J}\big|g(T^{-n}x)\big|^2 }d\mu(x).
\end{align}
Using Cauchy-Schwarz inequality, we can rewrite \eqref{inequa:1} as follows.
\begin{align}\label{inequa:2}
&\big(2(J-\bar{N})+1\big)\int 
\sup_{ \bar{N} \geq N  \geq 1}
\Big|\frac1{N}\sum_{n=1}^{N}a_n f(T^{n})g(T^{-n}) 
\Big| d\mu(x)  \nonumber\\
&\leq C \sqrt{\int \sum_{-J}^{J}\big|f(T^{n}x)\big|^2
\sum_{-J}^{J}\big|g(T^{-n}x)\big|^2 d\mu(x)}.\\
&\leq C \Bigg(\int \sum_{-J}^{J}\big|f(T^{n}x)\big|^2 \mu(x)\Bigg)^{\frac12}
\Bigg(\int \sum_{-J}^{J}\big|g(T^{-n}x)\big|^2 \mu(x)\Bigg)^{\frac12}.\nonumber\\
&\leq C (2J+1)\big\|f\big\|_2 \big\|g\big\|_2.
\end{align}
Letting $J \longrightarrow +\infty$, we get
\begin{align}\label{inequa:3}
&\int 
\sup_{ \bar{N} \geq N  \geq 1}
\Big|\frac1{N}\sum_{n=1}^{N}a_n f(T^{n})g(T^{-n}) 
\Big| d\mu(x)  \nonumber\\
&\leq C \big\|f\big\|_2\big\|g\big\|_2.
\end{align}
Now, letting $\bar{N}\longrightarrow +\infty$, we obtain the desired inequality. Thanks to Beppo Levi's monotone convergence theorem. The proof of Proposition \ref{CalderonP2} being left to the reader.
\QEDA

It is easy to formulate a $k$-multilinear's version of Proposition \ref{CalderonP} and \ref{CalderonP2}, for any  $k \geq 3$. For the proof of it, we refer to \cite[Appendix]{DemeterLT}. 

We end this section by recalling the maximal ergodic inequality for the shift map on $\Z$. As  we shall see, this maximal inequality is due essentially to Hardy-Littlewood.  Indeed, it is a direct consequence of Hardy-Littlewood maximal theorem (see for example \cite[Theorem 13.15, p.32]{Zygmund}. Nowadays, this later theorem has several proofs. Nevertheless, almost all recent books reproduced the simple and beautiful proof due to F. Riesz \cite{Riesz} in which  he used his ``rising sun lemma"  to get the weak Hardy-Littlewood maximal inequality. It turns out that this later inequality is equivalent to weak maximal ergodic inequality . For more details and historical facts, we refer to \cite[Section 2.6]{Simon}.

Here, from \cite{HL1}, we need exactly the following discrete maximal inequality.

\begin{lem}(\cite[Theorem 8.]{HL1})\label{HL}. Let $r>1$, and  $(a_n)$ be a sequence of positive number. Then, for any $J \geq 1$, we have 
	$$\sum_{x=1}^{J}\Big(\max_{m \leq x}\Big(\frac1{x-m+1}\sum_{n=m}^{x}a_n\Big)^r
	\leq 2\Big(\frac{r}{r-1}\Big)^r \sum_{n=1}^{J}a_n^r.$$
\end{lem}
A straightforward application of Lemma \ref{HL} yields the following strong maximal inequality.
\begin{lem}[Maximal inequality for the shift on integers.]\label{Mshift}Let $r>1$ and $f \in \ell^p(\Z)$. Then
$$\Big\|\sup_{ N  \geq 1} \Big|\frac1{N}\sum_{n=1}^{N}f(x+n)\Big| \Big\|_r
\leq 2^{\frac{1}{r}}\frac{r}{r-1}\Big\|f\Big\|_r.$$
\end{lem}
\section{Modern proof of Birkoff ergodic theorem }
The modern proof of Birkoff ergodic theorem is due essentially to Bourgain and it is based on the oscillation method \cite{B1},\cite{B2},\cite{B3}, \cite{B4} (see also \cite{Thouvenot}.). This later method goes back to Gaposhkin \cite{Gap1},\cite{Gap2}. But, it was developed  by Bourgain to prove several version of generalized ergodic pointwise theorem. It turns out that the oscillation method has a deep connection with Martingale theory, Harmonic analysis and BMO-$H^p$ spaces theory which is linked to Carlson measures (see for instance \cite{Jones} and \cite[Chap. 7, p.117]{Gra}). Here, using Calderon's correspondence principal \cite{Cal} (see also \cite{Be}) and a simple spectral argument, we will prove the following:
\begin{thm}\label{INeq1}Let $(X, \cB, \mu)$ be a probability space and $f \in L^2(X,\mu).$ 
Then, for any  increasing sequence $(N_k)$ of positive integers, for any $K \geq 1$, we have
$$ \sum_{k=1}^{K}\Big\|\sup_{\overset{N_k \leq N \leq N_{k+1}}{N\in S_\rho} }\Big|
\frac{1}{N}\sum_{n=1}^{N}f(T^nx)-\frac{1}{N_{k+1}}\sum_{n=1}^{N_{k+1}}f(T^nx)\Big|\Big\|_2 \leq C\sqrt{K} \big\|f\|_2,$$
where $C$ is an absolutely constant.
\end{thm}
As before the proof of Theorem \ref{INeq1} will follows from the following theorem.
\begin{thm}\label{INeq2}Let $\phi \in \ell^2(\Z)$ . Then, for any  increasing sequence $(N_k)$ of positive integers, for any $K \geq 1$, we have
\begin{align}\label{Bourgain:max}
\sum_{k=1}^{K}\Big\|\sup_{\overset{N_k \leq N \leq N_{k+1}}{N\in S_\rho} }\Big|
\frac{1}{N}\sum_{n=1}^{N}f(n+x)-\frac{1}{N_{k+1}}\sum_{n=1}^{N_{k+1}}f(x+n)\Big|\Big\|_{\ell^2(\Z)}
\nonumber \\ \leq C_\rho\sqrt{K} \big\|f\|_{\ell^2(\Z)},
\end{align}	
	 
where $C_\rho$ is an absolutely constant.
\end{thm}

The proof of Birkhoff ergodic Theorem will follows from Theorem \ref{INeq1} by virtue of the following corner stone lemma in the method of oscillation . We state it in a more general form than needed here, since we believe that this lemma is of independent interest and will be useful for additional applications\footnote{{``
		As I observed in \cite{B1}, this approach should rather be considered as a general method than the solution to some isolated problems.''\par} \hfill{J. Bourgain (28/02/1954-22/12/2018.)}}. For sake of completeness,  we present its proof (see also \cite[Proof of Theorem 5.]{B1}, \cite[p.204, and p.209]{B4},
	\cite[Theorem 4.]{Thouvenot}, \cite[Lemma 3.1]{Demeter} ). 

\begin{lem}[\textbf{Corner stone lemma of oscillation method.}]\label{Corner}
Let $p \geq 1$ and $(f_n)_{n \geq 1}$ be a sequence of measurable square-integrable functions on a $\sigma$-finite measure space $(X,\cA,\mu)$. Assume that there is a sequence $(C_k)_{k\geq 1}$ such that for any increasing sequence of positive integers $(N_k)$, for any $K \geq 1$, we have
\begin{enumerate}
\item $ \ds \sum_{k=1}^{K} \Big\|\sup_{N_k \leq N \leq N_{k+1} }\big|f_N-f_{N_{k+1}}\big|\Big\|_p
\leq C_K,$ and,
\item $ \ds \frac{C_K}{K} \longrightarrow 0$ as $K \rightarrow +\infty.$
\end{enumerate} 
Then the sequence $(f_n)_{n \geq 1}$ converges almost everywhere. 
\end{lem}
\begin{proof}We proceed by contradiction, assuming that $(f_n)_{n \geq 1}$ does not converge for almost all $x \in X$ and, since $\mu$ is $\sigma$-finite, we assume also that $\mu(X)$ is finite. Therefore, the Borel set $A$ of $x$ such that $(f_n(x))_{n \geq 1}$ does not converge has a positive measure which we denote by $\alpha$. For any $\varepsilon \in \Q_{+}^{*}$, put
	$$A_\epsilon=\Big\{x \in X~~:~~\forall N \in \N, \exists n,m >N \textrm{~s.t.~}
	\big|f_n(x)-f_m(x)\big|>\varepsilon\Big\}.$$
Then, $A_\varepsilon$ is a measurable set. Furthermore, $\mu(A_\varepsilon) \tend{\varepsilon}{0}\alpha>0,$ by our assumption. Therefore, for small enough $\varepsilon>0$ we can assert that $\mu(A_\varepsilon)>\frac{99.\alpha}{100}\setdef{\beta}$.

We proceed now to construct by induction a sequence $(N_k)$ for which we will prove that $(2)$ can not be satisfied. Let $N_1=1$ and assume that $N_k$ has been chosen. Therefore, by the definition of $A_\epsilon$, for all $x \in A_\epsilon$, there exist $n,m \geq N_k$ such that $\big|f_n(x)-f_m(x)\big|>\varepsilon.$ We thus deduce that for any $M>\max\big\{n,m\big\}$, there exist $s \in \N$ such that $N_k \leq s \leq M$ and   $$\big|f_s(x)-f_{M}(x)\big|>\frac{\varepsilon}{2}.$$
Since, otherwise, we would have
$$ \big|f_n(x)-f_m(x)\big| \leq \big|f_n(x)-f_M(x)\big|+
\big|f_m(x)-f_M(x)\big| < \varepsilon.$$
Let us put now 
$$B_{N_k,M}=\Big\{x \in A_\varepsilon~~:~~\sup_{N_k \leq N \leq M}
\big|f_N(x)-f_M(x)\big| > \frac{\varepsilon}{2}\Big\}.$$
Letting $M \longrightarrow +\infty$, we see that $\mu(B_{N_k,M}) \longrightarrow  \mu(A_\varepsilon)=\beta>0.$ We can thus choose $M$ large enough such that
$\mu(B_{N_k,M})>\frac{99.\beta}{100}\setdef\gamma.$ Put $N_{k+1}=M$. This finish the construction of the sequence $(N_k)$ for which we have, for any $k \in \N$, for any $x \in  B_{N_k,N_{k+1}}$, $ \sup_{N_k \leq N \leq N_{k+1} }\big|f_k(x)-f_{N_{k+1}}(x)\big| > \frac{\varepsilon}{2}.$
Applying now the Markov inequality trick to see that 
$$\Big\| \sup_{N_k \leq N \leq N_{k+1} }\big|f_k(x)-f_{N_{k+1}}(x)\big|\Big\|_p^p
\geq \frac{\mu(B_{N_k,N_{k+1}}) .\varepsilon^p}{2^p}>\frac{\gamma.\varepsilon^p}{2^p}.$$
Whence, for any $K \in \N^*$, 
\begin{align}\label{Contradi}
	\frac{1}{K} \sum_{k=1}^{K} \Big\|\sup_{N_k \leq N \leq N_{k+1} }\big|f_N-f_{N_{k+1}}\big|\Big\|_p >\sqrt[p]{\gamma}.\frac{\varepsilon}{2},
\end{align}
which contradict our assumption $(2)$ and the proof of the lemma is complete.
\end{proof} 
we need also the following lemma.
\begin{lem}\label{maxFejer}Let $g$ be a positive integrable function of the circle. Then, there is a constant $C_\rho$ such that
\begin{align}\label{maxi-HP1}
\int_{-\pi}^{\pi} \sup_{N \in S_\rho}\Big|\frac{1}{N}\Big(\sum_{n=1}^{N}e^{i n \theta}-\I_{[-\frac{\pi}{N},+\frac{\pi}{N}]}(\theta)\Big)\Big|^2 g(\theta) d\theta\leq C_\rho \int_{-\pi}^{\pi} g(\theta) d\theta.
\end{align}
\end{lem}
Indeed, we have
\begin{lem}\label{maxStrong}Let $g$ be a positive integrable function of the circle. Then, there is a constant $C_\rho$ such that
	\begin{align}\label{maxi-HP2}
	\sum_{N \in S_\rho}\int \Big|\frac{1}{N}\Big(\sum_{n=1}^{N}e^{i n \theta}-\I_{[-\frac{\pi}{N},+\frac{\pi}{N}]}(\theta)\Big)\Big|^2 d\theta \leq C_\rho \int_{-\pi}^{\pi} g(\theta) d\theta.
	\end{align}
\end{lem}
\begin{proof}By Jensen-Fubini theorem, we are going to prove the following
\begin{align}\label{maxi-HP3}
\int \sum_{N \in S_\rho} \Big|\frac{1}{N}\Big(\sum_{n=1}^{N}e^{i n \theta}-\I_{[-\frac{\pi}{N},+\frac{\pi}{N}]}(\theta)\Big)\Big|^2 g(\theta) d\theta \leq C_\rho \int_{-\pi}^{\pi} g(\theta) d\theta.
\end{align}

For that, let $\theta$ be given such that $\theta \neq 0$. Then, there exist $K>0$ such that $\rho^{-(K+1)} \leq \frac{|\theta|}{\pi} <\rho^{-(K+1)}.$ Write
\begin{multline}
\sum_{N \in S_\rho}\Big|\frac{1}{N}\Big(\sum_{n=1}^{N}e^{ i n \theta}
-\I_{[-\frac{\pi}{\N},+\frac{\pi}{N}]}(\theta)\Big)\Big|^2\\
=\sum_{\overset{N \in S_\rho}{\frac{\pi}{N} > |\theta|}}\Big|\frac{1}{N}\Big(\sum_{n=1}^{N}e^{i n \theta}-\I_{[-\frac{\pi}{N},+\frac{\pi}{N})}(\theta)\Big)\Big|^2\\+
\sum_{\overset{N \in S_\rho}{\frac{\pi}{N} \leq |\theta|}} \Big|\frac{1}{N}\Big(\sum_{n=1}^{N}e^{i n \theta}-\I_{[-\frac{\pi}{N},+\frac{\pi}{N})}(\theta)\Big)\Big|^2.
%\IEEEeqnarraynumspace
\end{multline}

To estimate the first sums, we write
\begin{align}\label{Sum1:eq1}
\sum_{\overset{N \in S_\rho}{\frac{\pi}{N} > |\theta|}}\Big|\frac{1}{N}\Big(\sum_{n=1}^{N}e^{i n \theta}-\I_{[-\frac{\pi}{N},+\frac{\pi}{N})}(\theta)\Big)\Big|^2
&= \sum_{\overset{N \in S_\rho}{\frac{\pi}{N} \geq |\theta|}}
\Big|\frac{1}{N}\Big(\sum_{n=1}^{N}e^{i n \theta}-1\Big)\Big|^2 \nonumber\\
&\leq \sum_{\overset{N \in S_\rho}{\frac{\pi}{N} \geq \theta}} |N \theta|^2,
\end{align}
since $|1-e^{ix}| \leq |x|,$ for any $x \in [-\pi,\pi).$ We thus get 
\begin{multline}\label{Sum1:eq2}
\sum_{\overset{N \in S_\rho}{\frac{\pi}{N} > \theta}}\Big|\frac{1}{N}\Big(\sum_{n=1}^{N}e^{i n \theta}-\I_{[-\frac{\pi}{N},+\frac{\pi}{N})}(\theta)\Big)\Big|^2
\\\leq \sum_{k \leq K+1}\rho^{-2(K-k)}=\sum_{k=1}^{K}\rho^{-2k}< \sum_{k=1}^{+\infty}\rho^{-2k}.
\end{multline}
We proceed now to the estimation of the second sum. Likewise, write
\begin{align}\label{Sum2:eq1}
\sum_{\overset{N \in S_\rho}{\frac{\pi}{N} \leq |\theta|}}\Big|\frac{1}{N}\Big(\sum_{n=1}^{N}e^{i n \theta}-\I_{[-\frac{\pi}{N},+\frac{\pi}{N})}(\theta)\Big)\Big|^2
&= \sum_{\overset{N \in S_\rho}{\frac{\pi}{N} < |\theta|}}
\Big|\frac{1}{N}\sum_{n=1}^{N}e^{i n \theta}\Big|^2 \nonumber\\
&= \sum_{\overset{N \in S_\rho}{\frac{\pi}{N} < |\theta|}} 
\Big|\frac{1-e^{i N \theta}}{N.(1-e^{i \theta})}\Big|^2,\nonumber\\
&= \sum_{\overset{N \in S_\rho}{\frac{\pi}{N} < |\theta|}} 
\frac{\bigg(\sin\big(\frac{N \theta}{2}\big)\bigg)^2}{N^2.\bigg(\sin\big(\frac{ \theta}{2}\big)\bigg)^2} \nonumber\\
& \leq \pi^2 \sum_{\overset{N \in S_\rho}{\frac{\pi}{N} < |\theta|}} 
\frac{1}{N^2.\theta^2}.
\end{align}
The inequality \eqref{Sum2:eq1} follows from the classical inequality: $$ \sin\bigg(\frac{x}{2}\bigg) \geq \frac{x}{\pi}, \textrm{~~for~~} 0<x<\pi.$$ 
Consequently, we obtain
\begin{align}\label{Sum2:eq2}
\sum_{\overset{N \in S_\rho}{\frac{\pi}{N} \leq |\theta|}}\Big|\frac{1}{N}\Big(\sum_{n=1}^{N}e^{i n \theta}-\I_{[-\frac{\pi}{N},+\frac{\pi}{N})}(\theta)\Big)\Big|^2
\leq \sum_{k \geq K}\rho^{-2(K+1-k)}\\
=\sum_{k \geq 1}\rho^{-2k}.\nonumber
%\IEEEeqnarraynumspace
\end{align}
Summarizing, we have proved that there exist a constant $C(\rho)$ such that 
$$\sup_{ \theta }\Big(\sum_{N \in S_\rho}\Big|\frac{1}{N}\Big(\sum_{n=1}^{N}e^{ i n \theta}
-\I_{[-\frac{\pi}{N},+\frac{\pi}{N})}(\theta)\Big)\Big|^2\Big) \leq C(\rho).$$
This combined with \eqref{maxi-HP3} yields the desired inequality. The proof of the Lemma is complete.
\end{proof}
We are now able to proceed to the proof of Theorem \ref{INeq2}.
\begin{proof}[\textbf{Proof of Theorem \ref{INeq2}.}] As in the proof of Lemma \ref{maxFejer}, we will prove that there exist a constant $C_\rho$ such that 
\begin{align}\label{Cal-Fejer:eq1}
\sum_{k=1}^{+\infty}\Big\|\sup_{\overset{N_k \leq N \leq N_{k+1}}{N\in S_\rho} }\big| 
\frac{1}{N}\sum_{n=1}^{N}f(n+x)-\frac{1}{N_{k+1}}\sum_{n=1}^{N_{k+1}}f(x+n)\big|\Big\|_{\ell^2(\Z)}^2\nonumber \\ \leq C_\rho\big\|f\big\|_{\ell^2(\Z)}^2.
\end{align}
By the spectral isomorphism theorem, let us denoted by $g_N$ the image of $\I_{[-\frac{\pi}{N},\frac{\pi}{N})}$, for each $N$.  Therefore,  we can rewrite \eqref{Cal-Fejer:eq1} as follows
\begin{multline*}
	\sum_{k=1}^{+\infty}\Big\|\sup_{\overset{N_k \leq N \leq N_{k+1}}{N\in S_\rho} }\big| 
	\frac{1}{N}\sum_{n=1}^{N}f(n+x)-g_N(x)+\\
	\big(g_N(x)-g_{N_{k+1}}(x)\big)-\frac{1}{N_{k+1}}\sum_{n=1}^{N_{k+1}}f(x+n)-g_{N_{k+1}}(x)\big)\big|\Big\|_{\ell^2(\Z)}^2 \\ 
	\leq C_\rho\big\|f\big\|_{\ell^2(\Z)}^2. \nonumber	
	%\IEEEeqnarraynumspace
\end{multline*}
But
\begin{align}\label{Lagrange}
\sum_{k=1}^{+\infty}\Big\|\sup_{\overset{N_k \leq N \leq N_{k+1}}{N\in S_\rho} }\big| 
\frac{1}{N}\sum_{n=1}^{N}f(n+x)-\frac{1}{N_{k+1}}\sum_{n=1}^{N_{k+1}}f(x+n)\big|\Big\|_{\ell^2(\Z)}^2 \nonumber
\\
\leq 3
	\sum_{k=1}^{+\infty}\sum_{x \in \Z}\Biggm(\sup_{\overset{N_k \leq N \leq N_{k+1}}{N\in S_\rho} }\Bigg(\Big| 
	\frac{1}{N}\sum_{n=1}^{N}f(n+x)-g_N(x)\Big|^2+ \nonumber\\
	\Big|g_N(x)-g_{N_{k+1}}(x)\Big|^2+
	\Big|\frac{1}{N_{k+1}}\sum_{n=1}^{N_{k+1}}f(x+n)-g_{N_{k+1}}(x)\Big|^2\Bigg)\Biggm), \nonumber\\	
	\leq 6\Bigg( \sum_{k=1}^{+\infty}\Big\|\sup_{\overset{N_k \leq N \leq N_{k+1}}{N\in S_\rho} }\Big| 
	\frac{1}{N}\sum_{n=1}^{N}f(n+x)-g_N(x)\big|\Big\|_{\ell^2(\Z)}^2 \nonumber\\+
	\sum_{k=1}^{+\infty}\Big\|\sup_{\overset{N_k \leq N \leq N_{k+1}}{N\in S_\rho} }\Big| 
	g_N(x)-g_{N_{k+1}}(x)\Big|\Big\|_{\ell^2(\Z)}^2\Bigg),
	%\IEEEeqnarraynumspace
\end{align}
by virtue of the following basic inequality
$$|a+b+c|^2 \leq 3\Big(a^2+b^2+c^2\Big), \textrm{~~for~~any~~} a,b,c \in \R,$$
and since, for any $x \in \Z$,
$$ \Big|\frac{1}{N_{k+1}}\sum_{n=1}^{N_{k+1}}f(x+n)-g_{N_{k+1}}(x)\Big| \leq \sup_{\overset{N_k \leq N \leq N_{k+1}}{N\in S_\rho} }\Big| 
\frac{1}{N}\sum_{n=1}^{N}f(n+x)-g_N(x)\Big|.$$
Now, observe that the first sum in \eqref{Lagrange} is bounded   by Lemma \ref{maxFejer}. Indeed, by the spectral isomorphism, we have
\begin{align*}
&\sum_{k=1}^{+\infty}\Bigg\|\sup_{\overset{N_k \leq N \leq N_{k+1}}{N\in S_\rho} }\Big| 
\frac{1}{N}\sum_{n=1}^{N}f(n+x)-g_N(x)\Big|\Bigg\|_{\ell^2(\Z)}^2\\
&\leq\sum_{k=1}^{+\infty}\sum_{\overset{ N_k \leq N \leq N_{k+1}}{N \in S_\rho}} \Big\|
\frac{1}{N}\sum_{n=1}^{N}e^{i n\theta}-\I_{[-\frac{\pi}{N},\frac{\pi}{N}]}\Big\|_{L^2(\sigma_f)}^2 
\\
&\leq \sum_{N \in S_\rho} \Big\|
\frac{1}{N}\sum_{n=1}^{N}\big(e^{i n\theta}-\I_{[-\frac{\pi}{N},+\frac{\pi}{N}]}(\theta)\big)\Big\|_{L^2(\sigma_f)}^2, 
\end{align*}
we further have
$$\sum_{k \geq 1}\sum_{\overset{ N_k \leq N \leq N_{k+1}}{N \in S_\rho}} \Big\|
\frac{1}{N}\sum_{n=1}^{N}\big(e^{i n\theta}-\I_{[-\frac{1}{N},+\frac{1}{N}]}(\theta)\big)\Big\|_{L^2(\sigma_f)}^2
\leq C_\rho \|f\|_{\ell^2(\Z)}^2,$$
since $\sigma_f$ is absolutely continuous  with respect to the Lebesgue measure.

We proceed now to estimate the second sum.  For that, notice that we have
%we will use the dyadic method as in Littlewood–Paley theory. Let $p$ such that $2^p \leq N_{k+1}-N_{k}< 2^{p+1},$ and put
%$$D_{j,k}=g_{N_k+j.2^{p-k}}-g_g_{N_k+(j+1).2^{p-k}}, j=
\begin{eqnarray*}
\sup_{\overset{N_k \leq N \leq N_{k+1}}{N\in S_\rho} }\Big|g_N(x)-g_{N_{k+1}}(x)\Big|
=\sup_{\overset{N_k \leq N \leq N_{k+1}}{N\in S_\rho} }\Big|g_N*\big(g_{N_k}-g_{N_{k+1}}\big)(x)\Big|.
\end{eqnarray*}
Indeed, by the spectral transfer, we have
\begin{eqnarray*}
\cF(g_N*\big(g_{N_k}-g_{N_{k+1}}\big))&=&
\I_{[-\frac{\pi}{N},+\frac{\pi}{N}]}
\Big(\I_{[-\frac{\pi}{N_k},+\frac{\pi}{N_k}]}-
\I_{[-\frac{\pi}{N_{k+1}},+\frac{\pi}{N_{k+1}]}}\Big)\\
&=& \I_{[-\frac{\pi}{N},+\frac{\pi}{N}]}-
\I_{[-\frac{\pi}{N_{k+1}},+\frac{\pi}{N_{k+1}}]}\\
&=&\cF(g_{N}-g_{N_{k+1}}).
\end{eqnarray*}
It follows that, 
\begin{align}\label{Mg}
\sum_{k \geq 1}\Big\|\sup_{\overset{N_k \leq N \leq N_{k+1}}{N\in S_\rho} }\Big|g_N(x)-g_{N_{k+1}}(x)\Big|\Big\|_{\ell^2(\Z)}^2\nonumber\\
= \sum_{k \geq 1}\Big\|\sup_{\overset{N_k \leq N \leq N_{k+1}}{N\in S_\rho} }\Big|g_N(x)*(g_{N_k}-g_{N_{k+1}})(x)\Big|\Big\|_{\ell^2(\Z)}^2 \nonumber\\
\leq C_\rho 
\sum_{k \geq 1} \big\|g_{N_k}-g_{N_{k+1}}\|_{\ell^2(\Z)}^2.
\end{align}
By the maximal inequality for the shift on integers (Lemma \ref{Mshift}). Indeed, for any 
$f \in \ell^2(\Z)$, we have 
\begin{align*}
\Big\|\sup_{N\in S_\rho} \Big|\frac{1}{N}\sum_{n=1}^{N}f(x+n)-g_N*f(x)\Big|\Big\|_{\ell^2(\Z)}\\
=\Big\|\sup_{N\in S_\rho} \Big|\frac{1}{N}\sum_{n=0}^{N-1}\I_{\{-1\}}\circ S^n*f-g_N*f\Big|\Big\|_{\ell^2(\Z)}\\
\leq \sum_{N \in S_\rho} \Big\|\frac{1}{N}\sum_{n=0}^{N-1}\I_{\{-1\}}\circ S^n*f-g_N*f\Big\|_{\ell^2(\Z)}
\\
=\sum_{N \in S_\rho} \Big\|\Big(\frac{1}{N}\sum_{n=0}^{N-1}e^{in \theta}-\I_{[-\frac{\pi}{N},+\frac{\pi}{N}]}\Big) \cF(f)(\theta)\Big\|_{L^2(\T)}
\leq C_\rho \big\|f\big\|_2,
\end{align*}

The last inequality is due to  Lemma \ref{maxStrong} and the spectral transfer isomorphism.

Now, observe that \eqref{Mg} implies the desired estimation about the second sum, since, again, by the spectral transfer,
\begin{align*}
\sum_{k \geq 1} \big\|g_{N_k}-g_{N_{k+1}}\|_{\ell^2(\Z)}^2&=
\sum_{k \geq 1} \big\|\I_{[-\frac{\pi}{N_k},+\frac{\pi}{N_k}]}-
\I_{[-\frac{\pi}{N_{k+1}},+\frac{\pi}{N_{k+1}}]}\|_{L^2(\sigma_f)}^2\\
&\leq \sigma_{f}\big([-\pi,\pi))=\big\|f\big\|_2.
\end{align*}
We thus conclude that \eqref{Cal-Fejer:eq1} holds, that is,

\begin{gather*}
	\sum_{k=1}^{+\infty}\sum_{\overset{ N_k \leq N \leq N_{k+1}}{N \in S_\rho}} \Big\|
	\frac{1}{N}\sum_{n=1}^{N}e^{i n\theta}-\frac{1}{N_{k+1}}\sum_{n=1}^{N_{k+1}}e^{i n\theta}\Big\|_{L^2(\sigma_f)}^2\\ \leq C_\rho.\sigma_{f}\big([-\pi,\pi)\big).
	%\IEEEeqnarraynumspace
	%\label{Spec:eq1}
\end{gather*}
It still to prove \eqref{Bourgain:max}. For that, we apply the triangle inequality and  to finish the proof, we apply Cauchy-Schwarz inequality to obtain
\begin{align*}
	&\sum_{k=1}^{K}\Big\|\sup_{\overset{N_k \leq N \leq N_{k+1}}{N\in S_\rho} }\Big|
	\frac{1}{N}\sum_{n=1}^{N}f(f(x+n))-\frac{1}{N_{k+1}}\sum_{n=1}^{N_{k+1}}f(x+n)\Big|\Big\|_2
	\\
	&\leq \sqrt{K} \Big(\sum_{k=1}^{K}\Big\|\sup_{\overset{N_k \leq N \leq N_{k+1}}{N\in S_\rho} }\Big|
	\frac{1}{N}\sum_{n=1}^{N}f(n+x)-\frac{1}{N_{k+1}}\sum_{n=1}^{N_{k+1}}f(n+x)\Big|\Big\|_2^2\Big)^\frac12\\
	&\leq C_\rho \sqrt{K} \|f\|_2.
\end{align*}
The proof of the theorem is complete. 
\end{proof}

\begin{rem}$\quad$
\begin{enumerate}[label=\alph*)]
	\item Notice that our proof can be considered as a simple proof of Theorem 2.6 in \cite{Jones} and Corollary 6.4.3 in \cite[chap.4, p.152]{Weber}. Notice also that the proof in the later reference is inspired by the Regularization Spectral Principal due to Talagrand.
\item  Let us further point out that the maximal ergodic inequality play a curial rule. 
\end{enumerate}
\end{rem}
\section{Proof of Bourgain bilinear ergodic theorem}
The proof of Bourgain  bilinear ergodic theorem will follows form the ergodic theorem. Indeed, By the fondamental Bourgain observation (see Equation (2.15) in \cite{Bourgain-D}),  for any $f,g \in \ell^2(\Z)$, we have  
\begin{align}\label{Bequa:1}
\frac1{N}\sum_{n=1}^{N} f(x+n)g(x-n)\nonumber\\ =
\int_{-\pi}^{\pi} \widehat{f}(\theta) \Bigg(\frac1{N}\sum_{n=1}^{N}g(x-n)e^{-i(x-n)\theta}\Bigg)
e^{2 i x\theta} d\theta.
%\IEEEeqnarraynumspace
\end{align}
Put $$g_{\theta}(x)=g(x)e^{i  x \theta},  \forall x \in \Z.$$ 
Then, for any $\theta \in [-\pi,+\pi)$, $(g_{\theta}(x)) \in \ell^2(\Z).$
Applying Jensen inequality, it follows that
\begin{align}\label{Boeq1}
	\Bigg|\frac1{N}\sum_{n=1}^{N} f(x+n)g(x-n)-\frac1{N_{k+1}}\sum_{n=1}^{N_{k+1}} f(x+n)g(x-n)\Bigg|\nonumber\\ =
	\Bigg|\int_{0}^{1} \widehat{f}(\theta) 
	\Bigg(\frac1{N}\sum_{n=1}^{N}g_{\theta}(x-n)-
	\frac1{N_{k+1}}\sum_{n=1}^{N_{k+1}}g_{\theta}(x-n)\Bigg)
	e^{2 i x \theta } d\theta\Bigg|.\\
	\leq 
	\int_{0}^{1} \big|\widehat{f}(\theta)\big| 
	\Bigg|\frac1{N}\sum_{n=1}^{N}g_{\theta}(x-n)-
	\frac1{N_k}\sum_{n=1}^{N_k}g_{\theta}(x-n)\Bigg|
	 d\theta.\nonumber
\end{align}
Squaring and using Cauchy-Schwarz inequality combined with Parseval inequality, we obtain
\begin{align}\label{Beq2}
\sup_{\overset{N_k \leq N \leq N_{k+1}}{N\in S_\rho} }	\Bigg|\frac1{N}\sum_{n=1}^{N} f(x+n)g(x-n)-\frac1{N_{k+1}}\sum_{n=1}^{N_{k+1}} f(x+n)g(x-n)\Bigg|^2\nonumber\\
	\leq \big\|f\|_{\ell^2(\Z)}^2 \int_{0}^{1}
	\sup_{\overset{N_k \leq N \leq N_{k+1}}{N\in S_\rho}}	
	\Bigg|\frac1{N}\sum_{n=1}^{N}g_{\theta}(x-n)-
	\frac1{N_{k+1}}\sum_{n=1}^{N_{k+1}}g_{\theta}(x-n)\Bigg|^2
	d\theta.
	%\IEEEeqnarraynumspace
	%\label{Beq2}
\end{align}	
Integrating, we rewrite \eqref{Beq2} as follows
%\begin{equation}
\begin{multline}\label{Bourgain:eq3}
\Bigg\|\sup_{\overset{N_k \leq N \leq N_{k+1}}{N\in S_\rho} }	\Big|\frac1{N}\sum_{n=1}^{N} f(x+n)g(x-n)\\- \frac1{N_{k+1}}\sum_{n=1}^{N_{k+1}}f(x+n)g(x-n)\Big|\Bigg\|_{\ell^2(\Z)}^2\nonumber\\
\leq \big\|f\|_2^2 \int_{0}^{1}
\Big\|\sup_{\overset{N_k \leq N \leq N_{k+1}}{N\in S_\rho}}	
\Big|\frac1{N}\sum_{n=1}^{N}g_{\theta}(x-n)\\-
\frac1{N_{k+1}}\sum_{n=1}^{N_{k+1}}	g_{\theta}(x-n)\Big|\Big\|_{\ell^2(\Z)}^2
d\theta.
%\IEEEeqnarraynumspace
\end{multline}
%\end{equation}
Applying \eqref{INeq2}, we get
\begin{equation}
\begin{split}
	\sum_{k=1}^{+\infty}\Bigg\|\sup_{\overset{N_k \leq N \leq N_{k+1}}{N\in S_\rho} }	\Bigg|\frac1{N}\sum_{n=1}^{N} f(x+n)g(x-n)\\-\frac1{N_{k+1}}\sum_{n=1}^{N_{k+1}} f(x+n)g(x-n)\Big|\Bigg\|_{\ell^2(\Z)}^2\\
\leq C_\rho \big\|f\|_{\ell^2(\Z)}^2 \big\|g\|_{\ell^2(\Z)}^2.
%\IEEEeqnarraynumspace
\label{Bourgain:eq4}
\end{split}
\end{equation}
It follows (again by the Cauchy–Schwarz inequality) that for any for any $K \geq 1$,
\begin{equation}
\begin{split}
\sum_{k=1}^{K}\Bigg\|\sup_{\overset{N_k \leq N \leq N_{k+1}}{N\in S_\rho} }	\Bigg|\frac1{N}\sum_{n=1}^{N} f(x+n)g(x-n)\\-\frac1{N_{k+1}}\sum_{n=1}^{N_{k+1}} f(x+n)g(x-n)\Big|\Bigg\|_{\ell^2(\Z)}\\
\leq C_\rho \sqrt{K} \big\|f\|_{\ell^2(\Z)}^2 \big\|g\|_{\ell^2(\Z)}^2.
%\IEEEeqnarraynumspace
\label{Bourgain:CS}
\end{split}
\end{equation}

In the same manner, by applying Theorem \ref{INeq2} combined with Caucy-Schwarz inequality, we can see that for any $K \geq 1$, for any $h \in \ell^2(\Z)$ such that $\big\|h\|_{\ell^2(\Z)}=1$, we have 

\begin{equation}
\begin{split}
\sum_{k=1}^{K} \Bigg| \Bigg \langle \sup_{\overset{N_k \leq N \leq N_{k+1}}{N\in S_\rho} }	\Bigg|\frac1{N}\sum_{n=1}^{N} f(x+n)g(x-n)\\-\frac1{N_{k+1}}\sum_{n=1}^{N_{k+1}} f(x+n)g(x-n)\Big|,h\Bigg\rangle_{\ell^2(\Z)}\Bigg|\\
\leq C_\rho \sqrt{K }\big\|f\|_{\ell^2(\Z)} \big\|g\|_{\ell^2(\Z)}.
%\IEEEeqnarraynumspace
\label{Bourgain:inner1}
\end{split}
\end{equation}

Now, using carefully similar arguments to that in the proof of Proposition \ref{CalderonP} and \ref{CalderonP2} combined with Cauchy-Schwarz inequality and Riesz representation theorem,  we conclude that, for any $F,G \in L^\infty(X)$,
%\begin{gather}
%\sum_{k=1}^{K}\Bigg\|\sup_{\overset{N_k \leq N \leq N_{k+1}}{N\in S_\rho} }	\Bigg|\frac1{N}\sum_{n=1}^{N} F(T^{n}x)G(T^{-n}x)-\frac1{N_{k+1}}\sum_{n=1}^{N_{k+1}} F(T^{n}x)G(T^{-n}x)\Big|\Bigg\|_1 \nonumber\\ \leq C_\rho  \sqrt{K} \big\|F\big\|_2 \big\|G\big\|_2.
%\end{gather}

%Applying again Cauchy-Scwharz inequality,
\begin{multline}
\sum_{k=1}^{K}\Bigg\|\sup_{\overset{N_k \leq N \leq N_{k+1}}{N\in S_\rho} }	\Bigg|\frac1{N}\sum_{n=1}^{N} F(T^{n}x)G(T^{-n}x)\\-\frac1{N_{k+1}}\sum_{n=1}^{N_{k+1}} F(T^{n}x)G(T^{-n}x)\Big|\Bigg\|_2 \\
\leq C_\rho \sqrt{K}\big\|F\|_{2} \big\|G\|_{2},
%\IEEEeqnarraynumspace
\label{Bourgain:eq5}
\end{multline}
and this achieve the proof.

\section{Proof of bilinear Ergodic Theorem along Polynomials and polynomials in primes (Theorem \ref{BPPET}).}\label{ProofBPPET}
For the proof of Theorem \ref{BPPET}, we need the following two lemmas.
\begin{lem}\label{POsc}Let $g \in \ell^2(\Z)$ and let $(N_k)$ be any sequence of positives integers such that $2N_k<N_{k+1},$ $k=1,2,\cdots.$. Then, for any non-constant polynomial $Q$ mapping natural numbers to themselves, 
for any $K \geq 1$, we have
\begin{align}\label{BPPET:OscP}
\sum_{k=1}^{K}\Big\|\sup_{\overset{N_k \leq N \leq N_{k+1}}{N\in S_\rho} }\Big|
\frac{1}{N}\sum_{n=1}^{N}g(x+Q(n))-\frac{1}{N_{k+1}}\sum_{n=1}^{N_{k+1}}g(x+Q(n))\Big|\Big\|_{\ell^2(\Z)}
\nonumber \\ \leq C_\rho C(K) \big\|g\|_{\ell^2(\Z)},
\end{align}		
and 
\begin{align}\label{BPPET:OscPP}
\sum_{k=1}^{K}\Big\|\sup_{\overset{N_k \leq N \leq N_{k+1}}{N\in S_\rho} }\Big|
\frac{1}{\pi_N}\sum_{\overset{p \leq N}{p ~~\textrm{prime}}}g(x+Q(p))-\frac{1}{\pi_{N_{k+1}}}\sum_{\overset{p \leq N_{k+1}}{p ~~\textrm{prime}}}g(x+Q(n))\Big|\Big\|_{\ell^2(\Z)}
\nonumber \\ \leq C_\rho C(K) \big\|g\|_{\ell^2(\Z)},
\end{align}
where $C_\rho$ is a constant depending only on $\rho$ and 
$ \ds \frac{C(K)}{K} \longrightarrow 0$ as $K \rightarrow +\infty.$ 	
\end{lem}
The second lemma is an extension of the Hardy-Littlewood maximal inequality (Lemma \ref{HL}), we state it as follows.

\begin{lem}\label{POmax}Let $r>1$ and $g \in \ell^r(\Z)$.  Then, for any non-constant polynomial $Q$ mapping natural numbers to themselves, we have
	\begin{align}\label{BPPET:maxP}
	\Big\|\sup_{ N \geq 1 }\Big|
	\frac{1}{N}\sum_{n=1}^{N}g(x+Q(n))\Big|\Big\|_{\ell^r(\Z)}
	\nonumber \\ \leq C \big\|g\|_{\ell^r(\Z)},
	\end{align}		
	and 
	\begin{align}\label{BPPET:maxPP}
	\Big\|\sup_{ N \geq 1 }\Big|
	\frac{1}{\pi_N}\sum_{\overset{p \leq N}{p ~~\textrm{prime}}}g(x+Q(p))\Big|\Big\|_{\ell^r(\Z)}
	\nonumber \\ \leq C\big\|g\|_{\ell^r(\Z)},
	\end{align}
	where $C$ is an absolute constant. 	
\end{lem}
The proof of the Lemma \ref{POsc} and Lemma \ref{POmax} is based essentially on the Hardy-Littlewood circle method. For their proof, we refer to \cite{B2}, \cite{B3}, \cite[Lemmas 4 and 5]{Nair} and the expository article \cite{Thouvenot}.

Now, we proceed to the proof of Theorem \ref{BPPET}. Applying the same arguments as in the proof of BBET, we obtain, for any $f,g \in \ell^2(\Z)$,
\begin{align}\label{Bbppequa:1}
\frac1{N}\sum_{n=1}^{N} f(x+P(n))g(x-P(n))\nonumber\\ =
\int_{-\pi}^{\pi} \widehat{f}(\theta) \Bigg(\frac1{N}\sum_{n=1}^{N}g(x-P(n))e^{-i(x-P(n))\theta}\Bigg)
e^{2 i x\theta} d\theta.
%\IEEEeqnarraynumspace
\end{align}
Whence
\begin{gather}\label{Bbpoeq1}
 \sup_{\overset{N_k \leq N \leq N_{k+1}}{N\in S_\rho} }\Bigg|\frac1{N}\sum_{n=1}^{N} f(x+P(n))g(x-P(n))\nonumber \\-\frac1{N_{k+1}}\sum_{n=1}^{N_{k+1}} f(x+P(n))g(x-P(n))\Bigg|\nonumber\\
\leq 
\int_{-\pi}^{\pi} \big|\widehat{f}(\theta)\big| \sup_{\overset{N_k \leq N \leq N_{k+1}}{N\in S_\rho} }
\Bigg|\frac1{N}\sum_{n=1}^{N}g_{\theta}(x-P(n))-
\frac1{N_k}\sum_{n=1}^{N_k}g_{\theta}(x-P(n))\Bigg|
d\theta,
\end{gather}
where  $g_{\theta}(x)=g(x)e^{i x \theta},  \forall x \in \Z.$ 

Now, let $K \geq 1$ and for each $k=1,2,\cdots,K$, let $h_k \in \ell^2(\Z)$ such that 
$\big\|h_k\big\|_{\ell^2(\Z)}=1.$ Then, by applying Lemma \ref{BPPET:OscP} combined with Cauchy-Schwarz inequality, we get 
\begin{gather}\label{BPPET:OscP2}
\sum_{k=1}^{K}\Bigg|\Big\langle\sup_{\overset{N_k \leq N \leq N_{k+1}}{N\in S_\rho} }\Big|
\frac{1}{N}\sum_{n=1}^{N}g(x-P(n))\nonumber\\-\frac{1}{N_{k+1}}\sum_{n=1}^{N_{k+1}}g(x-P(n))\Big|\overline{h_k(x)}\Big\rangle_{\ell^2(\Z)}\Bigg|
\nonumber \\ 
\leq \int_{-\pi}^{\pi}  \big|\widehat{f}(\theta)\big| \Big\langle\sup_{\overset{N_k \leq N \leq N_{k+1}}{N\in S_\rho} }
\Bigg|\frac1{N}\sum_{n=1}^{N}g_{\theta}(x-P(n))-\nonumber \\
\frac1{N_k}\sum_{n=1}^{N_k}g_{\theta}(x-P(n))\Bigg||h_k(x)|\Big\rangle
d\theta,\nonumber\\
\leq C_\rho C(K) \big\|f\|_{\ell^2(\Z)} \big\|g\|_{\ell^2(\Z)}.
\end{gather}	
  
This gives
\begin{gather}\label{BPPET:OscP3}
\sum_{k=1}^{K}\Bigg\|\sup_{\overset{N_k \leq N \leq N_{k+1}}{N\in S_\rho} }\Big|
\frac{1}{N}\sum_{n=1}^{N}g(x-P(n))-\frac{1}{N_{k+1}}\sum_{n=1}^{N_{k+1}}g(x-P(n))\big|\Bigg\|_{\ell^2(\Z)}
\nonumber \\ 
\leq C_\rho C(K) \big\|f\|_{\ell^2(\Z)} \big\|g\|_{\ell^2(\Z)},
\end{gather}
Thanks to Riesz representation theorem. We thus get that the convergence almost every holds for $F, G\in L^\infty(X,\mu)$. Applying the same machinery, we obtain %the strong maximal inequality and 
the proof of bilinear ergodic theorem along polynomials in primes.  To finish  the proof of Theorem \ref{BPPET}, we need to establish Theorem \ref{MPPET} 
%follows from the bilinear maximal inequality established by Lacey \cite{Lacey} (see also \cite{DemeterLT}).
\section{Proof of Theorem \ref{MPPET}}
Let us denote the bilinear maximal function by
$$M(f,g)=\sup_{ N  \geq 1}\Bigg|\frac1{N}\sum_{n=1}^{N} f(x+P(n))g(x-P(n))
\Bigg|.$$ 
Therefore, obviously, $M$ maps $\ell^\infty(\Z) \times \ell^\infty(\Z)$ to $\ell^\infty(\Z)$. We thus need to see that $M$ maps $\ell^1(\Z)  \times \ell^\infty(\Z) $ to $\ell^1(\Z)$. For that observe that if $f \in \ell^1(\Z), 
g \in \ell^r(\Z), r>1.$ Then, by applying the same reasoning as before combined with H\"{o}lder inequality, we have
\begin{gather}\label{MBIEq}
\sup_{ N \geq 1}\Bigg|
\frac{1}{N}\sum_{n=1}^{N}f(x+P(n))g(x-P(n))\Bigg|^r
\nonumber \\ 
\leq \big\|f\big\|_{\ell^1(\Z)}^r \Big(\int_{0}^{1}   
\sup_{ N \geq 1}\Bigg|\frac1{N}\sum_{n=1}^{N}g_{\theta}(x-P(n))\Bigg|
d\theta\Big)^r,\nonumber\\
\leq C_r \big\|f\|_{\ell^1(\Z)} \int_{0}^{1}   
\sup_{ N \geq 1}\Bigg|\frac1{N}\sum_{n=1}^{N}g_{\theta}(x-P(n))\Bigg|^r
d\theta.
\end{gather}
Therefore, by the maximal inequality, we obtain 
\begin{gather}\label{MBIEq2}
\Bigg\|\sup_{ N \geq 1}\Big|
\frac{1}{N}\sum_{n=1}^{N}f(x+P(n))g(x-P(n))\Big|\Bigg\|_r \nonumber\\
\leq C_r \big\|f\|_{\ell^1(\Z)} \big\|g\|_{\ell^r(\Z)}.
\end{gather}
The case $r=+\infty$ can be handled in the same manner. To finish the proof, we apply the bilinear interpolation.

For the general case, exploiting the fact that spectral type of the shift map on $\Z$ has a simple spectrum (see \cite[Theorem 4.2]{Be}), it can be seeing that the rotated version of Lemma \ref{POsc} and Lemma \ref{POmax} holds, that is, for any $\theta \in [-\pi,\pi)$, for any polynomial $R(n)$ taking integer values, for any $f,g \in \ell^2(\Z)$, we have
 \begin{align}\label{BPPET:OscPR}
 \sum_{k=1}^{K}\Big\|\sup_{\overset{N_k \leq N \leq N_{k+1}}{N\in S_\rho} }\Big|
 \frac{1}{N}\sum_{n=1}^{N}e^{iR(n)\theta}g(x+Q(n))-\nonumber \\\frac{1}{N_{k+1}}\sum_{n=1}^{N_{k+1}}e^{iR(n)\theta}g(x+Q(n))\Big|\Big\|_{\ell^2(\Z)}
 \nonumber \\ \leq C_\rho C(K) \big\|g\|_{\ell^2(\Z)},
 \end{align}		
 and 
 \begin{align}\label{BPPET:OscPPR}
 \sum_{k=1}^{K}\Big\|\sup_{\overset{N_k \leq N \leq N_{k+1}}{N\in S_\rho} }\Big|
 \frac{1}{\pi_N}\sum_{\overset{p \leq N}{p ~~\textrm{prime}}} e^{iR(p)\theta} g(x+Q(p))-\nonumber \\\frac{1}{\pi_{N_{k+1}}}\sum_{\overset{p \leq N_{k+1}}{p ~~\textrm{prime}}}e^{iR(p)\theta}g(x+Q(n))\Big|\Big\|_{\ell^2(\Z)}
 \nonumber \\ \leq C_\rho C(K) \big\|g\|_{\ell^2(\Z)},
 \end{align}
 where $C_\rho$ is a constant depending only on $\rho$ and 
 $ \ds \frac{C(K)}{K} \longrightarrow 0$ as $K \rightarrow +\infty.$
\begin{rem}A careful bilinear interpolation gives that the convergence almost everywhere holds for $r,r'\geq 1$ such that $1 \leq \frac1{r}+\frac{1}{r'}<\frac32.$ We notice that since the Fourier transform play a role of spectral isomorphism the range of $\frac1{r}+\frac{1}{r'}$ can be related to inequalities in Fourier analysis as established by Beckner \cite{Beckner}, many, the Young's inequalities on convolutions.  We hope also that this his direction will be explored in future.
	
\end{rem} 
\begin{ques} $\quad$
\begin{enumerate}[label=\Alph*--]
\item Our investigation leads us to conjecture that for any $S,R$ in the weak-closure of a given measure-preserving transformation  $T$, for any $f \in L^r(X,\mu), g \in L^{r'}(X)$ such that $ 1 \leq \frac{1}{r}+\frac{1}{r'} <\frac32$, the bilinear ergodic average 
$$\frac{1}{N}\sum_{n=1}^{N}f(S^nx)g(R^nx),$$
converge almost everywhere.  
\item We ask also whether the bilinear maximal inequality and the convergence almost everywhere  for the smooth class of functions holds, that is, for any $\phi,\psi$ two smooth functions on real line, can one prove or disprove that for a given measure-preserving transformation  $T$, for any $f \in L^r(X,\mu), g \in L^{r'}(X)$ such that $1 \leq \frac{1}{r}+\frac{1}{r'} < \frac32$, the bilinear average 
$$\frac{1}{N}\sum_{n=1}^{N}f(T^{[\phi(n)]}x) g(T^{[\psi(n)}]x),$$
converge almost everywhere, where as customary, $[x]$ is the integer part of $x$. 
\end{enumerate}   
\end{ques}
 
\begin{thank}
This  work was done while the author was in delegation at the CNRS, DR-19, and during a visit to the university of Sciences and Technology of China, Hefei, Anhui province, whose support and hospitality are gratefully acknowledged. The author would like to thanks also Karma Dajani and Mahesh Nerurkar for fruitful discussions on the subjects.  He is also thankful to Rutgers university  where the paper was
revised, for the invitation and hospitality.	
\end{thank}


\begin{thebibliography}{9999}
\bibitem{aaronson}
J. Aaronson, An introduction to infinite ergodic theory. Mathematical Surveys and Monographs, 50. American Mathematical Society, Providence, RI, 1997.

\bibitem{elabdal}
E. H. el Abdalaoui, On the homogeneous ergodic bilinear averages with M\"{o}bius and Liouville weights, arXiv:1706.07280v3 [math.CA].

\bibitem{Beckner}
W. Beckner, Inequalities in Fourier analysis, Annals of Mathematics, Second Series, 102 (1): 159-182.

\bibitem{Be}
A. Bellow, Transference principles in ergodic theory, Harmonic Analysis and Partial 
Differential Equations, edited by Michael Christ,
Carlos E. Kenig, and Cora Sadosky, Univ.
Chicago Press, Chicago, Illinois, (1999), 27-39.

\bibitem{B1}
J. Bourgain, On the maximal ergodic theorem for certain subsets of
the integers, Israel Journal of Mathematics, Vol. 61, (1) , 1988, 39-72.

\bibitem{B2}
J. Bourgain,  On the pointwise ergodic theorem on $L^p$ for arithmetic
sets, Israel Journal of Mathematics, Vol. 61, (1), 1988, 73-84.

\bibitem{B3}
J. Bourgain, An approach to pointwise ergodic theorems, Springer
L.N.M., Vol. 1317, 204-223.

\bibitem{B4}
J. Bourgain, Pointwise ergodic theorems on arithmetic sets, with
an appendix on return time sequences
(jointly with H. Fnrstenberg, Y. Katznelson, D. Ornstein), Publications Math\'ematiques de l'I.H.E.S., (69), 1989, 5-45.

\bibitem{Bourgain-D}
J. Bourgain, Double recurrence and almost sure convergence, J. Reine Angew. Math., 404
(1990), pp. 140-161.

\bibitem{Cal}
A. P. Calderon, Ergodic theory and translation-invariant operators, Proc.
Nat. Acad. Sci. USA 59 (1968), 349-353.

\bibitem{CH1}
L. Carleson, On convergence and growth of partial sums of Fourier series, Acta Math. 116 (1966), 135–157.

\bibitem{Cuny-W}
C. Cuny and M. Weber, \emph{Ergodic theorems with arithmetical weights}, to appear in Isr. J. Math., available at http://arxiv.org/pdf/1412.7640.

\bibitem{Demeter}
C. Demeter, Pointwise convergence of the ergodic bilinear Hilbert transform, Illinois J.Math. 51.4 (2007), pp. 1123-1158.

\bibitem{DemeterLT} 
C. Demeter, T. Tao, C. Thiele, Maximal multilinear operators. Trans. Amer. Math. Soc. 360 (2008), no. 9, 4989-5042. 

\bibitem{Etamedi}
N. Etemadi,  An elementary proof of the strong law of large numbers.
Z. Wahrsch. Verw. Gebiete 55 (1981), no. 1, 119-122.

\bibitem{Fbook}
H.  Furstenberg, Recurrence in  ergodic theory  and  combinatorial number theory,  M.  B.  Porter Lectures. Princeton University Press, Princeton,N.J., 1981.

\bibitem{Gap1}
V. F. Gaposhkin, A theorem on the convergence almost everywhere of a sequence of measurable
functions, and its applications to sequences of stochastic integrals. Math. USSR Sbornik 33 (1977), -è1-17.

\bibitem{Gap2}
V. F. Gaposhkin, Individual ergodic theorem for normal operators on $L^2$ , Funct. Anal. Appl., 15 (1981), 14-18.

\bibitem{Garcia}
A. Garcia, {\it Topics in almost everywhere convergence}, Markham Publishing Company, Chicago, 1970.

\bibitem{Gra}
L. Grafakos, \emph{ Modern Fourier analysis,} Third edition. Graduate Texts in Mathematics, 250. Springer, New York, 2014. 
 
\bibitem{CH4}
L. Grafakos, T. Tao, and E. Terwilleger, $L^p$ bounds for a maximal dyadic sum operator, Math. Z. 246 (2004), no. 1-2, 321-337.

\bibitem{KD2}
 P. Hall, C. C. Heyde, Martingale limit theory and its application. Probability and Mathematical Statistics. Academic Press, Inc. [Harcourt Brace Jovanovich, Publishers], New York-London, 1980.
 
\bibitem{HL1}
G. H. Hardy and J. E. Littlewood, A maximal theorem with function-theoretic applications, Acta Mathematica, 54 (1930), 81-116.

\bibitem{HL2}
G. H. Hardy, J. E. Littlewood and G. P\'{o}lya, Inequalities, second ed., Cambridge, at the University Press, 1952.

\bibitem{CH2}
R. A. Hunt, On the convergence of Fourier series, Orthogonal Expansions and their Continuous Analogues (Proc. Conf., Edwardsville, Ill., 1967), Southern Illinois Univ. Press, Carbondale, Ill., 1968, pp. 235–255.

\bibitem{Jones}
R. Jones, Z. Kaufman, J. Rosenblatt, and M. Wierdl, Oscillation in ergodic theory,
Ergodic Theory Dynam. Systems, 18 (1998), 889–935.

\bibitem{KW}
Y. Katznelson and B. Weiss, A simple proof of some ergodic theorems, Israel J. Math., 42 (1982), no. 4, 291–296.

\bibitem{KP}
M. Keane, K. Petersen, 
Easy and nearly simultaneous proofs of the ergodic theorem and maximal ergodic theorem. (English summary) Dynamics \& stochastics, 248–251,
IMS Lecture Notes Monogr. Ser., 48, Inst. Math. Statist., Beachwood, OH, 2006. 

%\bibitem{Lacey}
%M. Lacey, The bilinear maximal functions map into $L^p$ for $2/3 < p \leq 1$, Ann. of Math. (2) 151
%(2000), no. 1, 35--57.
 
\bibitem{CH3}
M. Lacey and C. Thiele, A proof of boundedness of the Carleson operator, Math. Res. Lett. 7 (2000), no. 4, 361–370.

\bibitem{MirekT}
M. Mirek and B. Trojan, Cotlar's ergodic theorem along the prime numbers,  J.  Fourier Ana. and App. 21 (2015), 822–848.

\bibitem{Nadkarni1}
M. G. Nadkarni, \emph{Basic ergodic theory,} Reprint of the 1997 original.
Texts and Readings in Mathematics, 15. Hindustan Book Agency, New Delhi, 2013.

\bibitem{Nadkarni}
M. G. Nadkarni, \emph{Spectral theory of dynamical systems,} Reprint of the 1998 original.
Texts and Readings in Mathematics, 15. Hindustan Book Agency, New Delhi, 2011.

\bibitem{Nair}
R. Nair, On polynomials in primes and J. Bourgain’s circle method approach to ergodic theorems II, Stud. Math. 105
(1993), no. 3, 207--233.

\bibitem{parry}
W. Parry, \emph{Topics in ergodic theory,} Reprint of the 1981 original. Cambridge Tracts in Mathematics, 75. Cambridge University Press, Cambridge, 2004.

\bibitem{Riesz}
F. Riesz, Sur un Th\'eor\`eme de maximum de Mm. Hardy et Littlewood, J. London Math. Soc., 7 (1), 1932 10-13.

\bibitem{Simon}
 B. Simon, Harmonic analysis. A Comprehensive Course in Analysis, Part 3. American Mathematical Society, Providence, RI, 2015. 
 
\bibitem{Thouvenot}
J-P. Thouvenot, La convergence presque s\^ure des moyennes ergodiques suivant certaines sous-suites d'entiers 
(d'apr\` es Jean Bourgain). (French) [Almost sure convergence of ergodic means along some subsequences of integers 
(after Jean Bourgain)] S\'eminaire Bourbaki, Vol. 1989/90. Ast\'erisque No. 189-190 (1990), Exp. No. 719, 133-153.

\bibitem{Weber}
M. Weber,  \emph{Dynamical Systems and Processes,} European Mathematical Society Publishing House, IRMA
Lectures in Mathematics and Theoretical Physics 14, 2009.

\bibitem{Wiener}
N. Wiener,\emph{The Fourier integral and certain of its applications,} Dover publications, (1958).

\bibitem{Wierdl}
M. Wierdl, \emph{ Pointwise ergodic theorem along the prime numbers.} Israel J. Math. 64 (1988), no. 3, 315-336.

\bibitem{KD1}
D.  Williams,  \emph{Probability with martingales,} Cambridge Mathematical Textbooks. Cambridge University Press, Cambridge, 1991. 

\bibitem {Zygmund}
A. ~Zygmund, {\it Trigonometric series {\rm vol. I \& II}}, second ed.,
Cambridge Univ. Press, Cambridge, 1959.


\end{thebibliography}
\end{document}